\documentclass{imanum}
\usepackage{graphicx}
\usepackage{color}
\jno{drnxxx}
\received{\today}
\revised{}

\def\C{{\mathbb C}}
\def\R{{\mathbb R}}

\def\d{{\mathrm d}}

\def\d{{\rm d}}

\newcommand\calA{{\mathcal A}}
\renewcommand\d{{\rm d}}

\newcommand\bfx{{\mathbf x}}

\newcommand\bfA{{\mathbf A}}

\newcommand\bfK{{\mathbf K}}
\newcommand\bfM{{\mathbf M}}

\newcommand\quadfor{\quad\hbox{ for }\quad}

\renewcommand\ss{s}

\newcommand{\diff}{\frac{\d}{\d t}}
\newcommand{\Ga}{\varGamma}

\newcommand{\Gat}{\Ga(t)}
\newcommand{\GT}{\mathcal{G}_T}
\newcommand{\laplace}{\Delta}

\newcommand{\nbg}{\nabla_{\Ga}}

\newcommand{\mat}{\partial^{\bullet}}

\newcommand{\nb}{\nabla}

\newcommand{\pa}{\partial}

\def \t {(t)}

\def \to {\rightarrow}
\newcommand{\vphi}{\varphi}
\def \nu {\textnormal{n}}

\newcommand\T{\mathtt{\bf T}}

\newcommand\surface{\Ga}

\newcommand\enodes\xs
\newcommand\nnodes\bfx
\newcommand\rnodes\xs

\newcommand\regmass\bfK
\newcommand\mass\bfM
\newcommand\stiff\bfA

\newcommand{\xs}{\bfx^\ast}

\renewcommand{\dim}{m}

\newcommand{\bbk}{\color{black}}
\newcommand{\ebk}{\color{black}}

\begin{document}

\title{Maximal regularity of backward difference time discretization for evolving surface PDEs and its application to nonlinear problems}
\shorttitle{Time discrete maximal regularity for evolving surface PDEs}

\author{%
{\sc
Bal\'{a}zs Kov\'{a}cs\thanks{Corresponding author. Email: balazs.kovacs@mathematik.uni-regensburg.de}} \\[2pt]
Faculty of Mathematics, University of Regensburg, \\
Universit\"atsstra\ss{}e 31, 93049 Regensburg, Germany\\[6pt]
{\sc and}\\[6pt]
{\sc Buyang Li}\thanks{Email: buyang.li@polyu.edu.hk}\\[2pt]
Department of Applied Mathematics, \\
The Hong Kong Polytechnic University, Hong Kong
}
\shortauthorlist{B.~Kov\'acs and B.~Li}

\maketitle

\begin{abstract}
{Maximal parabolic $L^p$-regularity of linear parabolic equations on an evolving surface is shown by pulling back the problem to the initial surface and studying the maximal $L^p$-regularity on a fixed surface. By freezing the coefficients in the parabolic equations at a fixed time and utilizing a perturbation argument around the freezed time, it is shown that backward difference time discretizations of linear parabolic equations on an evolving surface along characteristic trajectories can preserve maximal $L^p$-regularity in the discrete setting. 
The result is applied to prove the stability and convergence of time discretizations of nonlinear parabolic equations on an evolving surface, with linearly implicit backward differentiation formulae characteristic trajectories of the surface, for general locally Lipschitz nonlinearities. 
The discrete maximal $L^p$-regularity is used to prove the boundedness and stability of numerical solutions in the $L^\infty(0,T;W^{1,\infty})$ norm, which is used to bound the nonlinear terms in the stability analysis. 
Optimal-order error estimates of time discretizations in the $L^\infty(0,T;W^{1,\infty})$ norm is obtained by combining the stability analysis with the consistency estimates.}
{evolving surface, nonlinear parabolic equations, locally Lipschitz continuous, backward differentiation formulae, linearly implicit, maximal $L^p$-regularity, stability, convergence, maximum norm.}
\end{abstract}

\section{Introduction}
\label{sec:introduction}

This paper concerns discrete maximal $L^p$-regularity for parabolic partial differential equations (PDEs) on an evolving surface and its application to the analysis of evolving surface nonlinear parabolic equations of the form
\begin{equation}
\label{eq:PDE}
\left\{
\begin{alignedat}{3}
\mat u + u \, (\nb_{\Gat} \cdot v) - \laplace_{\Gat} u &= f(u,\nb_{\Ga\t} u) & \qquad & \textrm{ on } \Ga\t \,\,\,\mbox{for}\,\,\, t\in(0,T] , \\
u(\cdot,0) &= u^0 & \qquad & \textrm{ on } \Ga(0)=\Ga^0 ,
\end{alignedat}
\right.
\end{equation}
where $\Ga(t) \subset\R^{m+1}$ is an $m$-dimensional evolving closed hypersurface with given velocity $v$, \bbk $\nbg$ and $\laplace_{\Gat}$ denote the surface gradient and Laplacian, respectively, and $\mat$ denotes the material derivative, see Section~\ref{sec:the evolving surface - basics} for more details. \ebk The function $f$ is a given smooth nonlinear function of $u$ and $\nbg u$ but not necessarily globally Lipschitz continuous. 

\bbk We recall, that an abstract evolution equation $\dot u - A u = f$ on the Banach space $X$ is said to have the \emph{maximal $L^p$-regularity} if for any $f \in L^p([0,T],X)$ the following estimate holds:
\begin{equation*}
	\|\dot u\|_{L^p([0,T],X)} + \|A u\|_{L^p([0,T],X)} \leq C \|f\|_{L^p([0,T],X)} .
\end{equation*}
We refer to \cite{Amann,KunstmannWeis2004,LSU68,Lions,Lunardi} for the maximal regularity of parabolic partial differential equations.
\ebk 

Partial differential equations on evolving surfaces have received much attention in recent years due to their applications in physics and biology, see, e.g.~\cite{Barreira2011,CGG,DeckelnickElliott2001_existence,Elliott2010a,ElliottRanner_CahnHilliard,AlphonseElliottStinner_abstract,AlphonseElliottStinner_linear} 
and the survey articles \cite{DeckelnickDziukElliott_acta,DziukElliott_acta,BGN_survey}, which collect many applications and numerical results. {\color{black}
Energy-diminishing and structure-preserving methods for evolving surfaces driven by curvature flows have been developed in many articles; see \cite{Bao-Zhao-2021,BGN-2007-JCP,BGN-2008-JCP,Duan-Li-Zhang-2021,Dziuk-1991,Jiang-Li-2021}. 
}

The numerical analysis of \emph{nonlinear} evolving surface PDEs was considered in many articles (for papers on linear problems we refer to the references of the above surveys).  
Error estimates of semi-discrete finite element methods for the Cahn--Hilliard equation with nonlinearity $f(u)=u-u^3$ on an evolving surface were obtained by Elliott and Ranner \cite{ElliottRanner_CahnHilliard}. 
Error bounds of semi-discrete finite element methods for Cahn--Hilliard equations with a general locally Lipschitz continuous nonlinearity $f(u)$ were studied in \cite{CHsurf}. 
The discrete maximum principle was established for semi-linear evolving surface PDE systems with nonlinearity $f(u_1,\dotsc,u_k)$ in \cite{dmp_evolving_surf}, \bbk wherein \ebk the discrete maximum principle was applied to prove error bounds for full discretization with backward Euler scheme. 
Error estimates of full discretization with backward differentiation formulae (BDF methods) for quasi-linear and semi-linear evolving surface PDEs with a general nonlinear $f(u)$ were proved in \cite{KovacsPower_quasilinear}. 

In all of these articles the numerical solutions were proved bounded in $L^\infty$ in order to use the local Lipschitz continuity of the solution to obtain stability and convergence. However, the techniques 
cannot be applied to nonlinearities of the form $f(u,\nb_{\Ga\t}  u)$, which often appear in solution driven evolving surface PDEs and curvature flows; see \cite{soldriven,MCF,MCF_soldriven,Willmore,MCFgeneralised}. 
In this case, the error analysis typically requires proving the $W^{1,\infty}$-boundedness of numerical solutions in order to bound the nonlinear terms in the stability and convergence estimates. 

For linear parabolic equations on an evolving surface, a $W^{1,\infty}$-norm error estimate of semi-discrete finite element solutions was shown in \cite{KovacsPower_max}. For full discretization of parabolic equations on stationary surfaces by the backward Euler time-stepping scheme, $ L^\infty$-norm error estimates were shown in \cite{Kroener}. The approach of both papers depends on the linear structure of the equations and cannot be extended to nonlinear problems. 

For nonlinear evolving surface PDEs with nonlinearity of the form $f(u,\nb_{\Ga\t}  u)$ (e.g.~mean curvature flow, Willmore flow, etc. \cite{soldriven,MCF,MCF_soldriven,Willmore,MCFgeneralised}), the numerical methods are often based on discretizations along the flow, and the $W^{1,\infty}$-norm bounds of the error and numerical solution in the literature are obtained from optimal-order $H^1$-norm error bounds and the \emph{inverse estimates} for finite element functions. As a result, at least quadratic surface finite elements are required to be used, and a certain grid-ratio condition $\tau=O(h^\kappa)$ is also needed in the error analysis. 

\medskip
The objective of this paper is to establish the maximal $L^p$-regularity \bbk of non-linear evolving surface PDEs, and to show  maximal $L^p$-regularity \ebk and $W^{1,\infty}$-norm estimates of temporally semi-discrete BDF methods for evolving surface problems. We then use the established results to prove error estimates in the $W^{1,\infty}$-norm for the temporal semi-discretization of evolving surface PDEs with nonlinearity of the form $f(u,\nb_{\Ga\t}  u)$. 

In flat domains, discrete maximal $L^p$-regularity has been used for analysis of time discretizations \cite{AkrivisLiLubich2017,KunstmannLiLubich2016,AkrivisLi2017} and full discretization \cite{Cai-Li-Lin-Sun-2019} of nonlinear parabolic equations. 
The $W^{1,\infty}$-norm error estimates established in this paper are complementary to the analysis of fully discrete evolving surface FEMs in \cite{soldriven,MCF}, which cannot allow $h\rightarrow 0$ for a fixed stepsize $\tau$. 
Through analyzing the temporal semi-discretization and full discretization separately as in \cite{LiSun2013}, the $W^{1,\infty}$-norm error estimates of temporal semi-discretization in this paper would open the door towards the analysis of fully discrete evolving surface finite element methods with linear surface finite elements and without grid-ratio conditions. 

We will start by showing maximal $L^p$-regularity for \emph{linear} evolving surface PDEs on a stationary surface, and then extend this result to evolving surfaces using the pull-back map, a perturbation argument in time, and relying on classical PDE theory. \bbk A general abstract formulation based on the pull-back map was first used for evolving surface problems in \cite{AlphonseElliottStinner_abstract,AlphonseElliottStinner_linear} to show well-posedness and regularity results (maximal regularity estimates were not shown therein). \ebk 
Using the results of \cite{KovacsLiLubich2016} we will show that the maximal $L^p$-regularity property is preserved by BDF discretizations of linear problems on stationary surfaces and then extend the result to evolving surfaces by a perturbation argument in time.
We then use the obtained discrete maximal regularity results to prove error estimates for BDF time discretization of nonlinear parabolic problems on evolving surfaces. 
The nonlinearity $f :\R \times \R^{m+1} \rightarrow \R$ is assumed to be smooth but may not have global Lipschitz continuity. 


The paper is organised as follows:
In Section~\ref{sec:parab-probl-evolv}, we introduce the basic notation and function spaces on evolving surfaces, and the definition of weak solutions. 
In Section~\ref{sec:transl-init-surf}, we prove the maximal parabolic $L^p$-regularity of linear parabolic equations on evolving surfaces by pulling back the equations to the initial surface $\Ga^0$, then using a perturbation argument we extend this result to evolving surface problems. 
In Section~\ref{section:discrete maxreg}, we prove discrete maximal $L^p$-regularity of BDF time discretizations for linear problems. We first establish this result for parabolic equations on a stationary surfaces, and then extend the result to evolving surfaces by a perturbation argument in time, \bbk via a similar argument as \ebk in the time continuous case. 
In Section~\ref{Sec:TD-nonlinear}, we prove stability and error bound in the maximum norm for linearly implicit BDF methods for nonlinear parabolic problems on evolving surfaces.

\section{Notation}
\label{sec:parab-probl-evolv}

\subsection{The evolving surface}
\label{sec:the evolving surface - basics}
Let $m\ge 1$ be a fixed integer. 
We assume that the evolution of a hypersurface $\Ga\t \subset \R^{m+1}$ is given by a diffeomorphic flow map $X(\cdot,t):\Ga^0 \to \Ga(t)$, where $\Ga^0$ is a {\color{black}smooth} $m$-dimensional initial hypersurface and $X(\cdot,0)$ equals the identity map. We assume that $X(y,t)$ is smooth with respect to $(y,t)\in \Ga^0\times[0,T]$ and the inverse function $X^{-1}(x,t)$ is smooth with respect to $x\in\Ga(t)$ uniformly for $t\in[0,T]$. 


The {\color{black}material velocity (which is simply called velocity below)} and material derivative on the surface are respectively given, for $x=X(y,t) \in \Gat$ with $y\in\Ga^0$, by 
\begin{align}
\label{eq:ODE for positions}
v(x,t) =&\ \partial_t X(y,t) , \\
\intertext{and} 
\nonumber
\mat u(x,t)= &\ \frac{\d}{\d t}u(X(y ,t),t) .
\end{align} 
Let $\nu$ be the unit outward normal vector to the surface $\Ga(t)$. We denote by $\nabla_{\Ga(t)}u$ the tangential gradient of the function $u$, and denote by $\Delta_{\Ga(t)}u = \nabla_{\Ga(t)} \cdot \nabla_{\Ga(t)} u$ the Laplace--Beltrami operator acting on $u$. 

Since the Ricci curvature of the surface $\Ga(t)$ depends only on the second-order partial derivatives of the flow map, it follows that the Ricci curvatures of the surfaces $\Ga(t)$, $t\in[0,T]$, are uniformly bounded (possibly negative). 
For more details on all these basic concepts we refer to \cite{DziukElliott_ESFEM,DeckelnickDziukElliott_acta,DziukElliott_acta}, and the references therein. 
An unified abstract theory for evolving problems is found in \cite{ElliottRanner_unified}.

\subsection{Function spaces}

We briefly introduce some Sobolev and Bochner-type function spaces to be used in this paper. More details can be found in \cite{AlphonseElliottStinner_abstract,AlphonseElliottStinner_linear}.

On a given surface $\Ga(t)$ the conventional Sobolev space $W^{1,p}(\Ga(t))$, $1\leq p\leq \infty$, can be defined as 
\begin{equation*}
W^{1,p}(\Ga(t)) = \bigl\{ w \in L^{p}(\Ga(t)) \mid \nbg w \in L^{p}(\Ga(t))^{\dim+1} \bigr\} ,
\end{equation*}
and analogously $W^{k,p}(\Ga(t))$ can be defined for any integer $k\ge 0$; 
see \cite{demlow2009,DziukElliott_ESFEM}. 
For $k<0$ and $1<p<\infty$, 
$W^{k,p}(\Ga(t))$ denotes the dual space of 
$W^{-k,p'}(\Ga(t))$, with $1/p+1/p'=1$.

On the space-time manifold $\GT = \cup_{t\in(0,T)} (\Ga\t\times\{t\})$ the inhomogeneous Sobolev spaces, collecting time-dependent functions mapping into time-dependent spaces (note the subscript $t$), are defined by  
\begin{align}
L^p_t(0,T;W^{k,q}(\Ga(t))) = &\ \Big\{ w : \GT\rightarrow \R \mid w(\cdot,t)\in W^{k,q}(\Ga(t))\,\,\mbox{a.e.}\,\,t\in(0,T), \notag \\
&\,\,\,\, \phantom{ \big\{ w : \GT \to \R \mid } \;\, t \mapsto \|w(\cdot,t)\|_{W^{k,q}(\Ga(t))} \in L^p(0,T) \Big\} , \label{def-LpWkq} \\
W^{1,p}_t(0,T;W^{k,q}(\Ga(t)))  = &\ \Big\{ w \in L^p_t(0,T;W^{k,q}(\Ga(t)))  \mid 
\mat w \in L^p_t(0,T;W^{k,q}(\Ga(t)))  \Big\} , \label{def-W1pWkq}
\end{align}
with the standard notational convention $H^{k}(\Ga(t)) =
W^{k,2}(\Ga(t))$.  
Since the flow map $X(\cdot,t):\Ga^0 \to \Ga(t)$ is a diffeomorphism, it follows that, for $s=0,1$, $w \in W^{s,p}_t(0,T;W^{k,q}(\Ga(t))) $ if and only if $w(X(y,t),t) \in W^{s,p}(0,T;W^{k,q}(\Ga^0))$ as a function of $(y,t)\in\Ga^0\times(0,T)$.

\bbk
\begin{remark}
The function spaces defined in \eqref{def-LpWkq} and \eqref{def-W1pWkq} are the same as those in \cite{AlphonseElliottStinner_abstract,AlphonseElliottStinner_linear} or \cite{ACDE2021} 
(denoted therein by $L^p_{W^{k,q}}$ and $W^{1,p}_{W^{k,q}}$),  but employing a different notation here.  
We adopt the notation in \eqref{def-LpWkq}--\eqref{def-W1pWkq} in order to distinguish the following three different spaces: 
$$
	L^p_t(0,T;L^q(\Ga(t))) , \quad L^p(0,T;L^q(\Ga(s))), \quad \text{and} \quad L^p(0,T;L^q(\Ga^0)) , 
$$
which are all used in this article. 
In our notation, the subscript $t$ in $L^p_t(0,T;L^q(\Ga(t)))$ means that the $L^p$ norm is integrated against the evolving surface $\Ga\t$ which depends on $t$, and therefore 
$$L^p_t(0,T;L^q(\Ga(t)))=L^p_s(0,T;L^q(\Ga(s)))$$ is independent of symbol in the subscript, while the second space $L^p(0,T;L^q(\Ga(s)))$ does depend on the stationary surface $\Ga(s)$ and therefore depends on $s$.  
\end{remark}
\ebk

Analogously, we define $C^s_t([0,T];W^{k,q}(\Ga(t)))$ to be the space of functions $w :\GT\rightarrow \R$ such that $w(X(y,t),t)\in C^s([0,T];W^{k,q}(\Ga^0))$ as a function of $y$ and $t$. We will also use the following abbreviation:
$$
H^1(\GT)=L^{2}_t(0,T;H^1(\Ga(t))) \cap H^1_t(0,T;L^{2}(\Ga(t))) .
$$


\subsection{Weak formulation}


The weak formulation of \eqref{eq:PDE} reads as follows: A function $u  \in H^1(\GT)\cap L^\infty_t(0,T;W^{1,\infty}(\Ga(t)))$ is called a \emph{weak solution} of \eqref{eq:PDE} if the following equation holds for all $\varphi\in H^1(\GT)$ such that $\mat\varphi=0$: 
\begin{equation}
\label{eq:PDE weak form}
\diff \int_{\Gat}\!\! u \vphi + \int_{\Gat}\!\! \nb_{\Gat} u
\cdot \nb_{\Gat} \vphi =
\int_{\Gat}\!\! f(u,\nb_{\Gat} u) \, \vphi ,
\qquad\mbox{for almost every $t\in[0,T]$}.
\end{equation}  
In the case of linear problem the regularity condition $u\in L^\infty_t(0,T;W^{1,\infty}(\Ga(t)))$ can be relaxed as in \cite[Definition~4.1]{DziukElliott_ESFEM}. 

\subsection{Linearly implicit BDF methods for the nonlinear problem}

For a time stepsize $\tau > 0$ and for $t_n = n\tau$, $n=0,1,\dots,N$ with $N \leq T/\tau$, 
we consider the temporal semi-discretization of the weak formulation \eqref{eq:PDE weak form} by a linearly implicit $k$-step BDF, with $1\le k \leq 6$: Find an approximation $u^n$ to $u(\cdot,t_n)$ determined by  
\begin{equation}
\label{eq:linearly implicit BDF}
\begin{aligned}
\frac 1\tau \sum_{j=0}^k\delta_j \int_{\Ga(t_{n-j})}u^{n-j} \vphi^{n-j}  + 
\int_{\Ga(t_{n})} \nb_{\Ga(t_n)}u^{n}  \cdot  \nb_{\Ga(t_n)} \vphi^{n} = 
\int_{\Ga(t_{n})} f(\hat u^n,\nb_{\Ga(t_n)} \hat u^n) \vphi^{n} , \qquad n \geq k ,
\end{aligned}
\end{equation}
where $\delta_j$, $j=0,\dots,k$, are the coefficients of the $k$-step BDF method, $\vphi^j := \vphi(\cdot,t_j)\in H^1(\Ga(t_j))$ satisfying $\vphi^j\circ X(\cdot,t_j)=\vphi^{0}$ for $j=1,\dots,N$, and $\hat u^n$ are the extrapolated values given by 
\begin{equation}
\label{eq:extrapolation - hat u}
\hat u^n = \sum^{k-1}_{j=0}\gamma_j \mathcal{F}_X(t_n,t_{n-j-1}) u^{n-j-1} , \quad n \geq k ,
\end{equation}
with $\gamma_j$, $j=0,\dots,k-1$, being the coefficients of the polynomial 
$\gamma (\zeta) = \frac 1\zeta\big( 1- (1-\zeta)^k \big) = \sum_{i=0}^{k-1} \gamma_i\zeta^i $,  
and $\mathcal{F}_X(t,s)$ denoting the flow map from $\Ga(s)$ to $\Ga(t)$. In other words, for any function $w$ defined on $\Ga(s)$, $\mathcal{F}_X(t,s)w$ is a function on $\Ga(t)$, defined by
\begin{align}
(\mathcal{F}_X(t,s) w)(X(\cdot,t)) := w(X(\cdot,s)), \qquad \text{on} \ \Ga^0 .
\end{align}
For given $u^{n-j}$, $j=1,\dots,k$, the extrapolated values in $f(\hat u^n,\nabla_{\Ga(t_n)} \hat u^n)$ are known and therefore $u^n\in H^1(\Ga(t_n))$ can be determined uniquely by the linearly equation \eqref{eq:linearly implicit BDF}.

The coefficients $\delta_j$, $j=0,\dots,k$ of the $k$-step BDF method is determined by the generating polynomial
\begin{equation}
\label{eq:BDF generating functions}
\begin{aligned}
\sum_{\ell=1}^k \frac 1\ell  (1-\zeta)^\ell = \sum\limits^k_{j=0}\delta_j \zeta^j .
\end{aligned}
\end{equation}
It is known that the $k$-step BDF method has order $k$ and is A$(\alpha_k)$-stable for $1 \leq k \leq 6$, with $\alpha_1=\alpha_2=0.5\,\pi, \alpha_3=0.4779\,\pi, \alpha_4=0.4075\,\pi,
\alpha_5=0.2880\,\pi$ and $\alpha_6=0.0991\,\pi$, and unstable for $k>6$, see \cite[Section V.2]{HairerWannerII}. 
The A$(\alpha)$-stability is equivalent to $|\arg \delta (\zeta)|\leq \pi -\alpha$ 
for $|\zeta|\leq 1.$ 

In this paper we shall prove the following result by using the temporally semi-discrete maximal $L^p$-regularity of evolving surface PDEs. 

\begin{theorem}
	\label{theorem: convergence of BDF time discr}
	Let $p$ and $q$ be positive numbers satisfying $2/p+m/q<1$, \bbk and let $\Ga\t \subset \R^{m+1}$ be sufficiently smooth. \ebk 
	If the solution of the nonlinear evolving surface problem \eqref{eq:PDE} is sufficiently smooth, i.e. 
	$$u\in C_t([0,T];W^{2,q}(\Ga(t))) \, \cap \, C^{k+1}_t([0,T];L^q(\Ga(t))) ,$$ 
	and the error $e^n=u^n-u(\cdot,t_n)$ in the starting values is sufficiently small, satisfying 
	\begin{equation}\label{initial-error-BDF}
	\max_{0\le n \le k-1}\big(\tau^{-1}\|e^n\|_{L^q(\Ga(t_n))}+ \|e^n\|_{W^{2,q}(\Ga(t_n))}\big)
	\leq C_0 \tau^k \quad \mbox{for some constant $C_0 > 0$},
	\end{equation}
	then there exists a constant $\tau_0>0$ such that for $\tau<\tau_0$ the temporally semi-discrete solution given by \eqref{eq:linearly implicit BDF} satisfies the following estimates for $N \leq T/\tau$: 
	\begin{align}
	\label{Error-Time-discr-Lp}
	\bigg( \tau \sum_{n=1}^N \bigg\|\frac{e^{n}-\mathcal{F}_X(t_n,t_{n-1}) e^{n-1}}{\tau}\bigg\|_{L^q(\Ga(t_n))}^p 
	+ \tau \sum_{n=1}^N \|e^{n}\|_{W^{2,q}(\Ga(t_n))}^p  \bigg)^{\frac1p} \leq &\ C\tau^k ,\\
	\label{Error-Time-discr} 
	\max_{1\le n\le N} \|e^n\|_{W^{1,\infty}(\Ga(t_n))} 
	\leq &\ C\tau^k .
	\end{align}
	The constants $\tau_0$ and $C>0$ are independent of $\tau$ and $N$ {\rm(}but may depend on $C_0$, $T$ and $u${\rm)}. 
\end{theorem}

The proof of Theorem \ref{theorem: convergence of BDF time discr} is based on the temporally semi-discrete maximal $L^p$-regularity for evolving surface PDEs, which is used in to prove the $W^{1,\infty}$ boundedness of numerical solutions through error estimates in the discrete $L^p_t(0,T;L^q(\Ga(t)))$ norm and an inhomogeneous Sobolev embedding inequality; see Lemma \ref{lemma:discrete space-time Sobolev inequality}. 
The temporally semi-discrete maximal $L^p$-regularity result for evolving surface PDEs is established in Section \ref{section:discrete maxreg} based on the continuous version established in Section \ref{sec:transl-init-surf}.

\section{Maximal $L^p$-regularity of linear evolving surface PDEs}
\label{sec:transl-init-surf}

In this section, we present the maximal $L^p$-regularity for the linear evolving surface PDE problem 
\begin{equation}
\label{eq:ES-PDE-strong-form}
\left\{
\begin{alignedat}{3}
\mat u + u \nb_{\Gat} \cdot v - \laplace_{\Gat} u =&\ f & \qquad & \text{on}\,\,\, \Ga\t ,\\
u(\cdot,0) =&\ u^0 & \qquad & \textrm{on}\,\,\, \Ga(0)=\Ga^0 ,
\end{alignedat}
\right. 
\end{equation}
with a given function $f : \GT \to \R$ (independent of $u$). 

The weak formulation of the linear problem \eqref{eq:ES-PDE-strong-form} with $f \in L^\infty_t(0,T;L^\infty(\Ga(t)))$ reads: Find $u \in H^1(\GT) \hookrightarrow C_t([0,T];L^2(\Ga(t)))$ such that, for all $\vphi\in H^1(\GT)$ with $\mat\varphi=0$, 
\begin{equation}
\label{eq:ES-PDE-weak-form}
\diff \int_{\Gat} u \vphi + \int_{\Gat} \nb_{\Gat} u
\cdot \nb_{\Gat} \vphi = 
\int_{\Gat} f \vphi ,
\end{equation} 
holds for almost every $t\in[0,T]$, with the initial condition $u(\cdot,0)=u^0$ on $\Ga(0)=\Ga^0$. 

%
%

By pulling back \eqref{eq:ES-PDE-strong-form} onto the initial surface $\Ga^0$ (see Appendix), one can see that there exists a smooth function $a(y,t)$ and a smooth linear transform $B(y,t)$ on the tangent space $\T_y\Ga^0$ at $y\in\Ga^0$, such that $u\in H^1(\GT)$ is a solution of \eqref{eq:ES-PDE-weak-form} if and only if the function 
$$U(y,t):=u(X(y,t),t)$$
defines a solution $U\in H^1(\Ga^0\times(0,T))$ of the following weak problem, for all $\psi \in H^1(\Ga^0)$
\begin{equation}
\label{eq:ES-PDE-weak-form-2}
\begin{aligned}
&\frac{\d}{\d t} \int_{\Ga^0} a(y,t) U(y,t)\psi(y) 
+ \int_{\Ga^0} B(y,t)\nb_{\Ga^0}U(y,t)
\cdot \nb_{\Ga^0} \psi (y) = \int_{\Ga^0}a(y,t)F(y,t) \psi(y) ,
\end{aligned}
\end{equation}
holds for almost all $t\in[0,T]$, with $F(y,t):=f(X(y,t),t)$ and the initial condition $U(\cdot,0)=u^0$. 

Since the Riemannian metric on the evolving surface is positive definite, from the expressions \eqref{def-ayt} and \eqref{exp-Axit}--\eqref{eq:1} it follows that the functions $a(y,t)$ and $B(y,t)$ satisfy the following estimates: 
\begin{align}
\label{aXt} 
&C^{-1}\leq a(y,t)\leq C, &&
\forall\, y\in\Ga^0 , \,\,\forall\, t\in[0,T] , \\
\label{BXt}
&C^{-1}|\xi_y|^2\leq B(y,t)\xi_y\cdot\xi_y\leq C |\xi_y|^2 ,
&&\forall\, \xi_y\in \T_y\Ga^0,\,\,\forall\, y\in\Ga^0 ,\,\, \forall\, t\in[0,T] ,
\end{align}
where $C$ is some positive constant (depending only on the given flow map), \bbk see also \cite[Theorem~2.32]{AlphonseElliottStinner_abstract}, and \cite[Lemma~3.2]{Vierling2014}. \ebk 

Moreover, there exists a smooth and invertible linear operator $K(y,t): \T_y\Ga^0 \to \T_{X(y,t)}\Ga\t$ 
such that
\begin{align}\label{Def-K-operator}
(\nabla_{\Ga(t)}u)(X(y,t),t)=K(y,t)\nabla_{\Ga^0}U(y,t)
\quad\mbox{for}\,\,\, U(y,t)=u(X(y,t),t).
\end{align}
The expressions of $a(y,t), B(y,t)$ and $K(y,t)$ in a local chart can be found in Appendix. 
{\color{black}
We also refer to \cite{AlphonseElliottStinner_abstract,AlphonseElliottStinner_linear} for more details on abstract formulation and pull-back techniques.} 

Through integration by parts, we obtain that \eqref{eq:ES-PDE-weak-form-2} is the weak formulation of the parabolic PDE on the fixed initial surface $\Ga^0$:
\begin{align}
\label{PDE:Gamma}
\left\{
\begin{aligned}
\diff \Big(a(\cdot,t) U(\cdot,t) \Big) -\nb_{\Ga^0}\cdot\Big(B(\cdot,t)\nb_{\Ga^0} U(\cdot,t)\Big) =&\ a(\cdot,t) F(\cdot,t) 
\quad\mbox{on}\,\,\, \Ga^0\times(0,T), \\
U(\cdot,0) =&\ u^0 .
\end{aligned}
\right. 
\end{align}
\bbk The equivalence of strong solutions and the original and the pull-back equation is given in \cite[Theorem~2.32]{AlphonseElliottStinner_abstract}, while for a well-posedness result see Theorem~3.6 therein. 

\begin{remark}
For the weak formulations on the evolving surfaces, as in \eqref{eq:PDE weak form}, \eqref{eq:linearly implicit BDF} and \eqref{eq:ES-PDE-weak-form}, we consider a test function $\varphi$ with $\partial^{\bullet}\varphi=0$. 
For the weak formulation on a stationary surface, as in \eqref{eq:ES-PDE-weak-form-2}, we consider a test function $\psi$ only defined on the stationary initial surface. 
\end{remark}
\ebk 

We first show the following maximal $L^p$-regularity result on the stationary initial surface. 
\begin{theorem}[Maximal $L^p$-regularity in the Lagrangian coordinates]
	\label{theorem:MaxLp - Lagrangian}
	If $u^0=0$, then the solution $U$ of the pulled-back PDE \eqref{PDE:Gamma} obeys the following estimate:   
	\begin{equation}
	\label{MaxLp-0}
	\|\partial_tU\|_{L^{p}(0,T;L^q(\Ga^0))}
	+ \|U\|_{L^{p}(0,T;W^{2,q}(\Ga^0))}
	\leq C \| F\|_{L^{p}(0,T;L^q(\Ga^0))} \qquad \forall\, F \in L^{p}(0,T;L^q(\Ga^0)), 
	\end{equation}
	with $1<p,q<\infty$. The constant $C>0$ only depends on $\GT$. 
\end{theorem}

In order to prove Theorem \ref{theorem:MaxLp - Lagrangian}, we first show the following lemma --- the maximal $L^p$ regularity result for a problem with coefficients $a(y,s)$ and $B(y,s)$ frozen at some fixed time $s\in[0,T]$. 
\begin{lemma}
	\label{Lemma:time-independent}
	For any fixed $s\in(0,T]$, the solution of the
	problem
	\begin{align}
	\label{PDE:Gamma-s}
	\left\{
	\begin{aligned}
	\diff \Big(a(y,s) U(y,t) \Big)
	-\nb_{\Ga^0}\cdot\Big(B(y,s)\nb_{\Ga^0} U(y,t)\Big)
	=&\ F(y,t) && \textrm{ on }  \Ga^0 ,  \\
	U(y,0) =&\ 0 && \textrm{ on }  \Ga^0  , 
	\end{aligned}
	\right.
	\end{align}
	satisfies the following estimate: 
	\begin{align}
	\label{MaxLp-s}
	\begin{aligned}
	&\|\partial_tU\|_{L^{p}(0,s;L^q(\Ga^0))}
	+ \|U\|_{L^{p}(0,s;W^{2,q}(\Ga^0))}
	\leq C \| F\|_{L^{p}(0,s;L^q(\Ga^0))} ,
	\end{aligned}
	\end{align}
	where the constant $C>0$ may depend on $T$, but is independent of
	$s\in(0,T]$.
\end{lemma}

\begin{proof}
	Note that $U$ is a solution of \eqref{PDE:Gamma} if and only if $u(x,t)=U(X^{-1}(x,t),t)$ is a solution of \eqref{eq:ES-PDE-weak-form}. Similarly, $U$ is a solution of \eqref{PDE:Gamma-s} if and only if $u_s(x,t)=U(X^{-1}(x,s),t)$ is a solution of 
	\begin{equation*}
	\diff \int_{\Ga(s)} u_s \vphi + \int_{\Ga(s)} \nb_{\Ga(s)} u_s
	\cdot \nb_{\Ga(s)} \vphi = 
	\int_{\Ga(s)} f_s \vphi \quad\forall\,\varphi\in L^2(\Ga(s)) ,
	\end{equation*} 
	with $f_s(x,t)=a(X^{-1}(x,s),s)^{-1}F(X^{-1}(x,s),t)$. This is the weak form of the heat equation on the fixed surface $\Ga(s)$, i.e. 
	\begin{equation*}
	\left\{
	\begin{aligned}
	\partial_t u_s(x,t) - \Delta_{\Ga(s)} u_s(x,t) = &\ f_s(x,t)  &&\text{on } \Ga(s) , \\[5pt]
	u_s(x,0) = &\ 0  &&\text{on } \Ga(s)  .
	\end{aligned}
	\right. 
	\end{equation*}
	If we define $U_s(x,t)=e^{-t}u_s(x,t)$, then $U_s$ is a solution of the following shifted equation: 
	\begin{equation}
	\label{PDE:Gamma-ss}
	\left\{
	\begin{aligned}
	\diff U_s(x,t) - \Delta_{\Ga(s)} U_s(x,t) + U_s(x,t)	= &\ F_s(x,t)  &&\text{on } \Ga(s) , \\[5pt]
	U_s(x,0) = &\ 0  &&\text{on } \Ga(s)  , \\
	\end{aligned}
	\right. 
	\end{equation}
	with $F_s(x,t):=e^{-t} a(X^{-1}(x,s),s)^{-1}F(X^{-1}(x,s),t)$.
	In view of this equivalence it is sufficient to prove that the solution of the equation above satisfies
	the following maximal $L^p$-regularity estimate:
	\begin{equation}
	\label{MaxLp-ss}
	\|\partial_tU_s\|_{L^{p}(0,s;L^q(\Ga(s)))}
	+ \|U_s\|_{L^{p}(0,s;W^{2,q}(\Ga(s)))}
	\leq C \| F_s\|_{L^{p}(0,s;L^q(\Ga(s)))} .
	\end{equation}
	
	By choosing $\varepsilon=C_7^{-1}(\alpha-1)K^{-1}$ in
	\cite[Corollary 3.1]{Li-Yau},
	we see that the fundamental solution $G_s(t,x,y)$
	of the equation \eqref{PDE:Gamma-ss}
	satisfies the following Gaussian estimate for some constant $K_0$,
	\begin{equation}
	\label{Gaussian-ss}
	0\leq G_s(t,x,y) \leq K_0t^{-1} e^{-\frac{|x-y|^2}{K_0t}} ,
	\end{equation}
	where the constant $K_0$ is independent of $s\in[0,T]$, depending only on the lower bound of the Ricci curvature of the family of surfaces $\Ga(s)$, $s\in[0,T]$. Therefore, the conditions of \cite[Theorem 3.1]{HP97} are satisfied with $m=2$ and $p(r)=K_0e^{-r/K_0}$, which implies that the solution of \eqref{PDE:Gamma-ss} satisfies the maximal $L^p$-regularity estimate \eqref{MaxLp-ss}. This completes the proof of Lemma~\ref{Lemma:time-independent}. 
\end{proof}

\begin{proof}[Proof of Theorem \ref{theorem:MaxLp - Lagrangian}]
	The proof is based on a perturbation argument, and combines it with Lemma~\ref{Lemma:time-independent}. This idea originates form Savar\'e \cite[Proof of Theorem~2.1]{Savare1993}, and had proved to be useful many times since then, in particular in the context of discrete maximal $L^p$-regularity \cite{AkrivisLiLubich2017,KunstmannLiLubich2016,AkrivisLi2017}.
	
	We rewrite \eqref{PDE:Gamma} such that the coefficients on the left-hand side are fixed at time $s$ in exchange for extra terms on the right-hand side: 
	\begin{align}
	\label{PDE:Gamma-per}
	\begin{aligned}
	&\diff \Big(a(y,s) U(y,t) \Big)
	-\nb_{\Ga^0}\cdot\Big(B(y,s)\nb_{\Ga^0} U(y,t)\Big) \\
	&= a(y,t) F(y,t)
	+\diff \Big((a(y,s)-a(y,t)) U(y,t) \Big) 
	+\nb_{\Ga^0}\cdot\Big((B(y,t)-B(y,s)\nb_{\Ga^0} U(y,t)\Big) .
	\end{aligned}
	\end{align}
	By applying Lemma~\ref{Lemma:time-independent} to the equation above in the time interval $[0,s]$,
	we obtain
	\begin{align}\label{dtULpLq1}
	\begin{aligned}
	& \|\partial_tU\|_{L^{p}(0,s;L^q(\Ga^0))}
	+ \|U\|_{L^{p}(0,s;W^{2,q}(\Ga^0))} \\
	&\leq C \| a F\|_{L^{p}(0,s;L^q(\Ga^0))}
	+C\bigg\|\diff \Big((a(y,s)-a(y,t)) U(y,t) \Big)\bigg\|_{L^{p}(0,s;L^q(\Ga^0))} \\
	&\quad
	+C\bigg\|\nb_{\Ga^0}\cdot\Big((B(y,t)-B(y,s)\nb_{\Ga^0} U(y,t)\Big)\bigg\|_{L^{p}(0,s;L^q(\Ga^0))} \\
	&\leq C\|F\|_{L^{p}(0,s;L^q(\Ga^0))}
	+ C\big(\big\| |s-t|\partial_tU\big\|_{L^{p}(0,s;L^q(\Ga^0))}
	+ \big\|(\partial_ta) U\big\|_{L^{p}(0,s;L^q(\Ga^0))}\big) \\
	&\quad 
	+ C\big\| |s-t|U\big\|_{L^{p}(0,s;W^{2,q}(\Ga^0))} ,
	\end{aligned}
	\end{align}
	where we have used the smoothness of the functions $a(y,t)$ and $B(y,t)$ to derive the last inequality. 
	We define
	\begin{equation*}
	L(t') := \|\partial_tU\|_{L^{p}(0,t';L^q(\Ga^0))}^p \! + \! \|U\|_{L^{p}(0,t';W^{2,q}(\Ga^0))}^p 
	=\int_0^{t'} \!\!\! \big(\|\pa_t U(\cdot,t)\|_{L^q(\Ga^0)}^p \! + \! \|U(\cdot,t)\|_{W^{2,q}(\Ga^0)}^p\big) \d t , \\
	\end{equation*}
	with $\pa_t L(t) = \|\pa_t U(\cdot,t)\|_{L^q(\Ga^0)}^p + \|U(\cdot,t)\|_{W^{2,q}(\Ga^0)}^p$.
	Then raising \eqref{dtULpLq1} to power $p$ yields 
	\begin{align*}
	L(s)
	&\leq C\|F\|_{L^{p}(0,s;L^q(\Ga^0))}^p
	+ C\big(\big\| |s-t|\partial_tU\big\|_{L^{p}(0,s;L^q(\Ga^0))}^p
	+ \big\|U\big\|_{L^{p}(0,s;L^q(\Ga^0))}^p\big) \\
	&\quad 
	+ C\big\| |s-t|U\big\|_{L^{p}(0,s;W^{2,q}(\Ga^0))}^p \\
	&\leq C\|F\|_{L^{p}(0,s;L^q(\Ga^0))}^p
	+C\int_0^s  (s-t)^p\big(\|\partial_tU\|_{L^q(\Ga^0)}^p
	+\|U\|_{W^{2,q}(\Ga^0)}^p\big) \d t 
	+ C\big\|U\big\|_{L^{p}(0,s;L^q(\Ga^0))}^p \\
	&\leq C\|F\|_{L^{p}(0,s;L^q(\Ga^0))}^p
	+C\int_0^s (s-t)^p \partial_tL\t \d t 
	+ C\big\|U\big\|_{L^{p}(0,s;L^q(\Ga^0))}^p \\ 
	&= C\|F\|_{L^{p}(0,s;L^q(\Ga^0))}^p
	+C\int_0^s p (s-t)^{p-1} L(t)  \d t 
	+ C\big\|U\big\|_{L^{p}(0,s;L^q(\Ga^0))}^p 
	\quad \mbox{(integration by parts)} \\
	&\leq C\|F\|_{L^{p}(0,s;L^q(\Ga^0))}^p
	+CpT^{p-1}\int_0^s L(t)  \d t 
	+ C\big\|U\big\|_{L^{p}(0,s;L^q(\Ga^0))}^p ,
	\qquad \forall\, s\in[0,T].
	\end{align*}
	Since $U(\cdot,0)=0$, it follows that 
	\begin{align*}
	\big\|U\big\|_{L^{p}(0,s;L^q(\Ga^0))}^p
	=\int_0^s \bigg\|\int_0^t \partial_tU(\cdot,\eta)\d \eta \bigg\|_{L^q(\Ga^0)}^p\d t 
	&\le \int_0^s \int_0^t t^{p-1} \|\partial_tU(\cdot,\eta)\|_{L^q(\Ga^0)}^p \d \eta \d t \\
	&\le T^{p} \int_0^s L(t) \d t .
	\end{align*}
	Substituting this into the estimate of $L(s)$ above, we obtain
	\begin{align*}
	L(s)
	&\leq C\|F\|_{L^{p}(0,s;L^q(\Ga^0))}^p
	+C\int_0^s L(t)  \d t  ,
	\quad\forall\, s\in[0,T].
	\end{align*}
	Then applying Gronwall's inequality yields 
	\begin{equation*}
	\begin{aligned}
	L(T)
	&\leq C\|F\|_{L^{p}(0,T;L^q(\Ga^0))}^p ,
	\end{aligned}
	\end{equation*}
	which implies 
	\begin{equation*}
	\begin{aligned}
	&\|\partial_tU\|_{L^{p}(0,T;L^q(\Ga^0))}
	+ \|U\|_{L^{p}(0,T;W^{2,q}(\Ga^0))}
	\leq C \| F\|_{L^{p}(0,T;L^q(\Ga^0))} .
	\end{aligned}
	\end{equation*}
	This completes the proof of Theorem \ref{theorem:MaxLp - Lagrangian}.
\end{proof}

It is crucial to note that usually maximal parabolic regularity estimates do not hold for the interesting cases $p$ or $q = \infty$. However, such type of estimate can be obtained via a space-time Sobolev embedding result as in the planar case \cite{KunstmannLiLubich2016}.

\begin{lemma}
	\label{lemma:space-time Sobolev inequality}
	Let $w \in W^{1,p}_t(0,T;L^q(\Ga(t))) \cap L^p_t(0,T;W^{2,q}(\Ga(t)))$ for $2/p + m/q < 1$ with $\Ga(t) \subset \R^{m+1}$ and $w(\cdot,0) = 0$. Then the following bound holds:
	\begin{equation}
	\label{eq:space-time Sobolev inequality}
	\|w\|_{L^\infty_t0,T;W^{1,\infty}(\Ga(t)))} \leq c_{p,q} \big( \|\mat w\|_{L^p_t(0,T;L^q(\Ga(t))) } + \|w\|_{L^p_t(0,T;W^{2,q}(\Ga(t))) } \big) .
	\end{equation}
\end{lemma}
\begin{proof}
{\color{black} 
	In view of the definition $U(y,t)=u(X(y,t),t)$, combining this with the definition of the material derivative we have
	\begin{equation*}
	\mat u(x,t) = \diff u(X(y,t),t) = \partial_tU(y,t) , \quadfor x=X(y,t) \in \Gat ,
	\end{equation*}
	which then yields
	\begin{align*}
	\|\mat u(\cdot,t)\|_{L^q(\Gat)}^q 
	= \int_{\Ga^0} \Big| \diff u(X(y,t),t) \Big|^q a(y,t) \d y 
	= \int_{\Ga^0} \Big| \partial_t U(y,t) \Big|^q a(y,t) \d y .
	\end{align*}
Since $C^{-1}\le a(y,t) \le C$ as shown in \eqref{aXt}, we obtain the following norm equivalence: 
		\begin{equation}
		\label{eq:norm equivalence for time derivatives}
		\begin{aligned}
		C_0 \|\partial_tU(\cdot,t)\|_{L^q(\Ga^0)} \leq &\ \|\mat u(\cdot,t)\|_{L^q(\Gat)} \leq  C_1 \|\partial_tU(\cdot,t)\|_{L^q(\Ga^0)} 
		\quad \mbox{for}\,\,\, 1<q<\infty .
		\end{aligned}
		\end{equation}
Inequality \eqref{eq:space-time Sobolev inequality} follows from the same bounds for flat domains proved in \cite[Lemma~3.1]{KunstmannLiLubich2016} and the norm equivalence in \eqref{eq:norm equivalence for time derivatives}. 
}
\end{proof}

On the right-hand side of \eqref{eq:space-time Sobolev inequality}, the expression appears as in Theorem~\ref{theorem:MaxLp - surface coordinates} below.

We now translate the continuous maximal $L^p$-regularity estimate of Theorem \ref{theorem:MaxLp - Lagrangian} back to the evolving surface functional analytic setting, formulating it for the linear evolving surface PDE \eqref{eq:ES-PDE-strong-form}. 
\begin{theorem}[Maximal $L^p$-regularity in the surface coordinates]
	\label{theorem:MaxLp - surface coordinates}
	Let $u^0=0$ and let $f \in L^p_t(0,T;L^q(\Ga(t)))$ with $1<p,q<\infty$. Then the solution $u$ of the linear evolving surface PDE problem \eqref{eq:ES-PDE-strong-form} obeys the following maximal parabolic $L^p$-regularity estimate:   
	\begin{equation}
	\label{MaxLp on Gamma}
	\|\mat u\|_{L^p_t(0,T;L^q(\Ga(t))) } + \|u\|_{L^p_t(0,T;W^{2,q}(\Ga(t))) }  \leq C \|f\|_{L^p_t(0,T;L^q(\Ga(t))) } .
	\end{equation}
	Furthermore, if $f \in L^p_t(0,T;L^q(\Ga(t)))$ such that $2/p + m/q < 1$, then the following bound holds:
	\begin{equation}
	\label{Linfty on Gamma}
	\|u\|_{L^\infty_t0,T;W^{1,\infty}(\Ga(t)))} 
	\leq c_{p,q} \, C  \|f\|_{L^p_t(0,T;L^q(\Ga(t))) } .
	\end{equation}
	The constants $C > 0$ and $c_{p,q} > 0$ may depend on the final time $T$ {\rm(}increasing function of $T)$.
\end{theorem}

\begin{proof}
Since $u(x,t)=U(X^{-1}(x,t),t)$ and the inverse flow map $X^{-1}(x,t)$ is smooth with respect to $x$, through the composition of the two functions $U(\cdot,t)$ and $X^{-1}(\cdot,t)$ we obtain 
		\begin{equation}
		\label{eq:norm equivalence for spatial derivatives}
		\begin{aligned}
		C_0' \|U(\cdot,t)\|_{W^{2,q}(\Ga^0)} \leq & 
		\|u(\cdot,t)\|_{W^{2,q}(\Ga(t))} \leq  C_1' \|U(\cdot,t)\|_{W^{2,q}(\Ga^0)} .
		\end{aligned}
		\end{equation}
In view of the norm equivalence relations in \eqref{eq:norm equivalence for time derivatives} and \eqref{eq:norm equivalence for spatial derivatives}, inequality \eqref{MaxLp on Gamma} is an immediate consequence of \eqref{MaxLp-0}. The second estimate in \eqref{Linfty on Gamma} directly follows from the first (with suitable $p$ and $q$) and Lemma~\ref{lemma:space-time Sobolev inequality}.
\end{proof}

\section{Maximal $L^p$-regularity of BDF methods for linear evolving surface PDEs}
\label{section:discrete maxreg}

\subsection{BDF methods for evolving surface PDEs}
\label{section:BDF}
We consider a $k$-step BDF method for the weak formulation \eqref{eq:ES-PDE-weak-form}, with $1\leq k \leq 6$. For a time stepsize $\tau > 0$ and for $t_n = n\tau \leq T$ we determine a semi-discrete approximation $u^n$ to $u(\cdot,t_n)$ by
\begin{align}
\label{linearEq:linearly-BDF}
\begin{aligned}
\frac 1\tau \sum_{j=0}^k\delta_j \int_{\Ga(t_{n-j})}u^{n-j} \vphi^{n-j}  + 
\int_{\Ga(t_{n})} \nabla_{\Ga(t_{n})}u^{n}  \cdot  \nabla_{\Ga(t_{n})} \vphi^n = 
\int_{\Ga(t_{n})} f(\cdot,t_n) \vphi^n , 
\quad n \geq k ,
\end{aligned}
\end{align}
for all $\vphi^n\in\Gamma(t_n)$ such that $\varphi^n\circ X(\cdot,t_n)=\varphi^0$, $n=0,1,\dots,N$.
The starting values $u^i$, $i=0,\dotsc,k-1$, are assumed to be given. 
They can be precomputed using either a lower order method with smaller step sizes, or an implicit Runge--Kutta method.
%

The order of the method for $k \leq 5$ is known to be retained for parabolic abstract evolution equations \cite{AkrivisLubich_quasilinBDF} (see the recent paper \cite{AkrivisChenYuZhou} for $k = 6$), for linear parabolic evolving surface PDEs \cite{LubichMansourVenkataraman_bdsurf}, for general partial differential equations \cite{AkrivisLiLubich2017,AkrivisLi2017} using discrete maximal regularity, and also for discretizations of the mentioned geometric flows \cite{soldrivenBDF,MCF,MCF_soldriven,MCFgeneralised}.

\subsection{Discrete maximal $L^p$-regularity}

Similarly as in the time-continuous case, the equations \eqref{linearEq:linearly-BDF} are pulled back to the initial surface, then the functions $U^n(y)=u^n(X(y,t_n))$, $n \geq k$, are solutions of the following problem: 
\begin{align}
\label{BDF: ES-PDE-weak-form-2}
\begin{aligned}
\frac 1\tau \sum_{j=0}^k\delta_j \int_{\Ga^0}a(\cdot,t_{n-j}) U^{n-j} \psi 
+ \int_{\Ga^0} B(\cdot,t_n)\nb_{\Ga^0} U^n 
\cdot \nb_{\Ga^0} \psi   = \int_{\Ga^0} a(\cdot,t_n) F^n \psi , \quad n \geq k,
\end{aligned}
\end{align}
with $F^n=f(X(\cdot,t_n),t_n)$ for all $\psi\in H^1(\Ga^0)$. 
Note that \eqref{BDF: ES-PDE-weak-form-2} is the weak formulation of the following PDE problem on the initial surface $\Ga^0$:
\begin{equation}
\label{eq:BDF - pulled back}
\frac{1}{\tau}\sum_{j=0}^k\delta_j  a(\cdot,t_{n-j}) U^{n-j} 
-\nb_{\Ga^0}\cdot\Big(B(\cdot,t_n)\nb_{\Ga^0} U^n\Big)
= a(\cdot,t_n) F^n .
\end{equation}
For a sequence of functions $(U^n)_{n=0}^N$ defined on the initial surface $\Ga^0$, we denote by 
\begin{equation}
\label{eq:discrete BDF derivative}
\overline\partial_\tau U^n = \frac 1\tau \sum_{j=0}^k\delta_j U^{n-j}, \quadfor n= 0,1,2,\dotsc,
\end{equation}
the discrete time derivative defined by the $k$-step BDF method 
(where $U^n=0$ for $n\le -1$). 
Similarly as in \cite{AkrivisLiLubich2017,KunstmannLiLubich2016}, \bbk considering piecewise constant functions, \ebk for a sequence of functions $(v^n)_{n=0}^N \subset W^{j,q}(\Ga^0)$ we define the norm 
\begin{align}
\|(v^n)_{n=0}^N\|_{L^p(W^{j,q}(\Ga^0))}
= 
\left\{
\begin{aligned}
&\bigg(\tau\sum_{n=0}^N\|v^n\|_{W^{j,q}(\Ga^0)}^p\bigg)^{\frac1p} ,
&&\mbox{for}\,\,\, p\in[1,\infty),\\[5pt] 
&\max_{0\leq n\leq N}\|v^n\|_{W^{j,q}(\Ga^0)} , &&\mbox{for}\,\,\, p = \infty .
\end{aligned}
\right.
\end{align}

In terms of these notations, we have the following result. 
\begin{theorem}[Discrete maximal $L^p$-regularity for BDF methods]
	\label{theorem:MaxLp - LagrangianBDF}
	Let $F^n \in L^q(\Ga^0)$ for $n \tau \leq T$, with $1<q<\infty$. 
	Then the solutions of \eqref{BDF: ES-PDE-weak-form-2} obey the following estimate for $1<p<\infty$ and $N \leq T/\tau$, 
	\begin{align}
	\label{BDFMaxLp-s}
	\begin{aligned}
	&\ \bigg\|\bigg(\frac{U^n-U^{n-1}}{\tau}\bigg)_{n=0}^N \bigg\|_{L^{p}(L^q (\Ga^0))}
	+ \|(U^n)_{n=0}^N\|_{L^{p}(W^{2,q} (\Ga^0))}\\
	\leq &\ C \|(F^n)_{n=k}^N\|_{L^{p}(L^q (\Ga^0))}
	+ C(\tau^{-1}\|(U^i)_{i=0}^{k-1}\|_{L^{p}(L^q (\Ga^0))}+ \|(U^i)_{i=0}^{k-1}\|_{L^{p}(W^{2,q} (\Ga^0))}) ,
	\end{aligned}
	\end{align}
	where the term $\|(F^n)_{n=k}^N\|_{L^{p}(L^q(\Ga^0))}$ vanishes when $N<k$, and the constant $C>0$ is independent of $\tau$ and $N$, but may depend on $p,q$ and on $T$. 
\end{theorem}

\begin{remark}{\upshape
		Since $\tfrac1\tau \sum_{j=0}^k \delta_j U^{n-j}$ can be expressed as $\tfrac1\tau \sum_{j=0}^k \delta_j U^{n-j} =  \sum_{j=0}^{k-1} \sigma_j \tfrac{U^{n-j}-U^{n-j-1}}{\tau}$ with some coefficients $\sigma_j$, $j=0,\dots,k-1$, it follows that 
		\begin{align}
		\label{BDFUn-dtUn}
		\big\|\big(\overline\partial_\tau U^n\big)_{n=0}^N \big\|_{L^{p}(L^q (\Ga^0))} 
		\leq &\ C \bigg\|\bigg(\frac{U^n-U^{n-1}}{\tau}\bigg)_{n=0}^N \bigg\|_{L^{p}(L^q (\Ga^0))} \\
		&\  + C(\tau^{-1}\|(U^i)_{i=0}^{k-1}\|_{L^{p}(L^q (\Ga^0))}+ \|(U^i)_{i=0}^{k-1}\|_{L^{p}(W^{2,q} (\Ga^0))})  \qquad \mbox{for} \ N \ge k.
		\end{align}
		Therefore, \eqref{BDFMaxLp-s} implies
		\begin{align*}
		\begin{aligned}
		&\ \big\|\big(\overline\partial_\tau U^n\big)_{n=0}^N \big\|_{L^{p}(L^q (\Ga^0))} 
		+ \|(U^n)_{n=0}^N\|_{L^{p}(W^{2,q} (\Ga^0))}\\
		\leq &\ C \|(F^n)_{n=k}^N\|_{L^{p}(L^q (\Ga^0))}
		+ C(\tau^{-1}\|(U^i)_{i=0}^{k-1}\|_{L^{p}(L^q (\Ga^0))}+ \|(U^i)_{i=0}^{k-1}\|_{L^{p}(W^{2,q} (\Ga^0))}) ,
		\end{aligned}
		\end{align*}
	}
\end{remark}

{\it Proof of Theorem~\ref{theorem:MaxLp - LagrangianBDF}.}$\,\,$
The proof is divided into two parts. 
In Part I, we show discrete maximal $L^p$-regularity for stationary surfaces, via a series of auxiliary lemmas. 
In Part II, we extend the result to evolving surface problems by using the result from Part I and a perturbation argument in time.  The proof has a parallel structure to the proof of Theorem \ref{theorem:MaxLp - Lagrangian} in the continuous case.

{\it Part I: Discrete maximal $L^p$-regularity on a stationary surface.}


We start by recalling that \cite[Example 3.7.5]{ABHN} implies that the self-adjoint negative definite operator $\calA_2(s)=\Delta_{\Ga(s)}-1:D(\calA_2(s))\rightarrow L^2(\Ga(s))$ with domain $D(\calA_2(s))=H^2(\Ga(s))$ generates a bounded analytic semigroup $\{E_2(z,s)\}_{z\in\Sigma_{\pi/2}}$ on $L^2(\Ga(s))$, where $\Sigma_{\pi/2} = \{z\in\C: |{\rm arg}(z)| <\pi/2\}$ (a sector of angle $\pi/2$ on the complex plane). We have seen that the kernel $G_s(t,x,y)$ of the semigroup $\{E_2(t,s)\}_{t>0}$ satisfies the Gaussian estimate \eqref{Gaussian-ss}. Consequently, the kernel $G_s(t,x,y)$ has an analytic extension to the right half-plane, satisfying 
(see \cite[pp.\ 103]{Davis89}) 
\begin{equation}
\label{GKernelE0}
|G_s(z,x,y)| \leq C_\theta |z|^{-\frac{m}{2}}
e^{-\frac{|x-y|^2}{C_\theta |z|}}, 
\quad \forall\, z\in \Sigma_{\theta},\,\,\forall\, x,y\in\Ga(s) ,
\quad\forall\,\theta\in(0,\pi/2), 
\end{equation}
where the constant $C_\theta$ is independent of $s\in[0,T]$ (depending only on $K_0$ and $\theta$). In other words, the rotated operator $-e^{i\theta}\calA_2(s)$ satisfies the condition 
of \cite[Theorem 8.6, with $m=2$ and 
$g(r)=C_\theta e^{-r^2/C_\theta}$]{KunstmannWeis2004}; also see \cite[Remark 8.23]{KunstmannWeis2004}.
As a consequence of \cite[Theorem 8.6]{KunstmannWeis2004}, 
$\{E_2(z,s)\}_{z\in\Sigma_{\pi/2}}$ extends to an analytic 
semigroup $\{E_q(z,s)\}_{z\in\Sigma_{\pi/2}}$ on $L^q(\Ga(s))$, for all $1<q<\infty$, $R$-bounded in the sector $\Sigma_\theta$ for all $\theta\in(0,\pi/2)$, where the $R$-bound is independent of $s\in[0,T]$ (depending only on $C_\theta$ and $q$). We refer to \cite{KunstmannWeis2004} for a general discussion on $R$-boundedness. 
%

Then Weis' characterization of maximal $L^p$-regularity \cite[Theorem 4.2]{Weis2001_2} immediately implies the following result. 
\begin{lemma}[$R$-boundedness of the resolvent, \cite{Weis2001_2}]
	\label{RbdRes}
	Let $\calA_q(s)$ denote the generator of the semigroup $\{E_q(t,s)\}_{t>0}$, with $1<q<\infty$. 
	Then the set $\{z(z-\calA_q(s))^{-1}\,:\, z\in \varSigma_\theta \}$ is $R$-bounded for all $\theta\in(0,\pi)$, and the $R$-bound depends only on $K_0$, $\theta$ and $q$. 
\end{lemma}

\begin{remark}{\upshape
		The operator $\calA_2(s)^{-1}:L^2(\Ga(s)) \rightarrow D(\calA_2(s))$ naturally extends to $\calA_q(s)^{-1}:L^q(\Ga(s))\rightarrow D(\calA_q(s))$ with  $\calA_q(s)$ denoting the operator $\Delta_{\Ga(s)}-1$ with domain $W^{2,q}(\Ga(s))$; see \cite[Appendix]{JinLiZhou2017-NM}. 
	}
\end{remark}

Lemma~\ref{RbdRes} and \cite[Theorem 4.1 and 4.2, Remark 4.3]{KovacsLiLubich2016} imply the following maximal $L^{p}$-regularity result for BDF discretizations of the heat equation on a stationary surface.

\begin{lemma}
	\label{Lemma:time-independent for BDF}
	For any given $s\in[0,T]$, the solutions $U_s^n$ of the discretized equation 
	\begin{equation}
	\label{BDF:Gamma-ss}
	\frac{1}{\tau}\sum_{j=0}^k\delta_j  U_s^{n-j} - \Delta_{\Ga(s)}U_s^n	= F_s^n , \qquad \text{on} \ \ \Ga(s) ,\quad k\leq n \leq N,
	\end{equation}
	with starting values $U_s^i$, $i=0,\dots,k-1$, and $F_s^n:=a(X^{-1}(\cdot,s),s)^{-1}F^n(X^{-1}(\cdot,s))$, satisfy the following discrete maximal $L^{p}$-regularity estimate for $0\le N\le T/\tau$: 
	\begin{align}
	\label{BDF-MaxLp-ss}
	&\ \bigg\|\bigg(\frac{U_s^n-U_s^{n-1}}{\tau}\bigg)_{n=0}^N \bigg\|_{L^{p}(L^q(\Ga^0))}
	+ \|(U_s^n)_{n=0}^N\|_{L^{p}(W^{2,q}(\Ga^0))} \nonumber \\
	\leq &\ C \| (F_s^n)_{n=k}^N\|_{L^{p}(L^q(\Ga^0))} 
	+ C(\tau^{-1}\|(U_s^i)_{i=0}^{k-1}\|_{L^{p}(L^q(\Ga^0))}+ \|(U_s^i)_{i=0}^{k-1}\|_{L^{p}(W^{2,q}(\Ga^0))})  
	\end{align}
	where $U^{-1}_s:=0$ and the term $\| (F_s^n)_{n=k}^N \|_{L^{p}(L^q(\Ga^0))}$ vanishes if $N<k$, and the constant $C>0$ is independent of $N$, $\tau$ and $s\in[0,T]$.
\end{lemma}
\begin{proof}
	Since $\calA_q(s)=\Delta_{\Ga(s)}-1$ in Lemma \ref{RbdRes}, we first consider the following problem (with an additional low-order term): 
	\begin{equation}
	\label{BDF:Gamma-ss2}
	\frac{1}{\tau}\sum_{j=0}^k\delta_j  U_s^{n-j}
	-\Delta_{\Ga(s)}U_s^n +U_s^n
	= F_s^n , \qquad \text{on} \ \Ga(s), \quad \text{for} \ k\tau \leq n\tau \leq T ,
	\end{equation}
	For the problem \eqref{BDF:Gamma-ss2}, Lemma~\ref{RbdRes} and \cite[Theorem 4.1 and 4.2, Remark 4.3]{KovacsLiLubich2016} imply, for $N \le T/\tau$,  
	\begin{align}
	\label{BDF-MaxLp-ssa}
	&\ \bigg\|\bigg(\frac{1}{\tau}\sum_{j=0}^k\delta_j  U_s^{n-j}\bigg)_{n=k}^N \bigg\|_{L^{p}(L^q(\Ga^0))}
	+ \|(U_s^n)_{n=k}^N\|_{L^{p}(W^{2,q}(\Ga^0))} \nonumber \\
	\leq &\ C \| (F_s^n)_{n=k}^N\|_{L^{p}(L^q(\Ga^0))} 
	+ C(\tau^{-1}\|(U_s^i)_{i=0}^{k-1}\|_{L^{p}(L^q(\Ga^0))}+ \|(U_s^i)_{i=0}^{k-1}\|_{L^{p}(W^{2,q}(\Ga^0))}) ,
	\end{align}
	In the case $U_s^n=0$ for $0\le n\le k-1$, it is known that the following inequality holds (cf. \cite[inequality (3.4)]{Li-2021-IMA}): 
	\begin{align*}
	&\bigg\|\bigg(\frac{U_s^n-U_s^{n-1}}{\tau}\bigg)_{n=k}^N\bigg\|_{L^{p}(L^q(\Ga^0))}
	\le 
	C\bigg\|\bigg(\frac{1}{\tau}\sum_{j=0}^k\delta_j  U_s^{n-j}\bigg)_{n=k}^N \bigg\|_{L^{p}(L^q(\Ga^0))} .
	\end{align*}
	In the case $U_s^n\neq 0$ for $0\le n\le k-1$ one can modify the above inequality by including the starting values, i.e. 
	\begin{align}\label{Dtau-Dtauk}
	&\bigg\|\bigg(\frac{U_s^n-U_s^{n-1}}{\tau}\bigg)_{n=k}^N \bigg\|_{L^{p}(L^q(\Ga^0))}
	\le 
	C\bigg\|\bigg(\frac{1}{\tau}\sum_{j=0}^k\delta_j  U_s^{n-j}\bigg)_{n=k}^N \bigg\|_{L^{p}(L^q(\Ga^0))}
	\!\! + C\tau^{-1}\|(U_s^n)_{n=0}^{k-1}\|_{L^{p}(L^q(\Ga^0))} . 
	\end{align}
	Inequalities \eqref{BDF-MaxLp-ssa} and \eqref{Dtau-Dtauk} imply that 
	\begin{equation}
	\label{BDF-MaxLp-ss - with extra term}
	\begin{aligned}
	&\ \bigg\|\bigg(\frac{U_s^n-U_s^{n-1}}{\tau}\bigg)_{n=k}^N \bigg\|_{L^{p}(L^q(\Ga^0))}
	+ \|(U_s^n)_{n=k}^N\|_{L^{p}(W^{2,q}(\Ga^0))} \nonumber \\
	\leq &\ C \| (F_s^n)_{n=k}^N\|_{L^{p}(L^q(\Ga^0))} 
	+ C(\tau^{-1}\|(U_s^i)_{i=0}^{k-1}\|_{L^{p}(L^q(\Ga^0))}+ \|(U_s^i)_{i=0}^{k-1}\|_{L^{p}(W^{2,q}(\Ga^0))}) .
	\end{aligned}
	\end{equation}
	This proves \eqref{BDF-MaxLp-ss} for the problem \eqref{BDF:Gamma-ss2} 
	(which has an additional lower order term).

	We now turn to prove \eqref{BDF-MaxLp-ss} for the original problem \eqref{BDF:Gamma-ss}, by removing the lower order term using a perturbation argument.
	To this end, we add the low-order term $U_s^n$ to both sides of \eqref{BDF:Gamma-ss}, 
	\begin{equation}
	\label{BDF:Gamma-ss3}
	\frac{1}{\tau}\sum_{j=0}^k\delta_j  U_s^{n-j}
	-\Delta_{\Ga(s)}U_s^n+U_s^n
	= F_s^n+U_s^n , \qquad \text{on} \ \Ga(s) \quad k \leq n \leq T/\tau ,
	\end{equation}
	and apply the estimate \eqref{BDF-MaxLp-ss - with extra term} to the problem above. We therefore obtain, for $N \leq T/\tau$,
	\begin{align}
	\label{eq:maxreg bound - pre power p}
	&\ \bigg\|\bigg(\frac{U_s^n-U_s^{n-1}}{\tau}\bigg)_{n=0}^N \bigg\|_{L^{p}(L^q(\Ga^0))}
	+ \|(U_s^n)_{n=0}^N\|_{L^{p}(W^{2,q}(\Ga^0))} \nonumber \\
	\leq &\ C \| (U_s^n)_{n=0}^N\|_{L^{p}(L^q(\Ga^0))} 
	+C \| (F_s^n)_{n=k}^N\|_{L^{p}(L^q(\Ga^0))} \nonumber \\
	&\ + C(\tau^{-1}\|(U_s^i)_{i=0}^{k-1}\|_{L^{p}(L^q(\Ga^0))}+ \|(U_s^i)_{i=0}^{k-1}\|_{L^{p}(W^{2,q}(\Ga^0))}) .
	\end{align}
	Using a H\"older inequality and $U_s^{-1}=0$, yields the estimate \begin{align*}
	\|(U_s^n)_{n=0}^N\|_{L^{p}(L^q(\Ga^0))}^p
	= \tau\sum_{n=0}^N \|U_s^n\|_{L^q(\Ga^0)}^p
	= &\ \tau\sum_{n=0}^N \bigg\|\tau \sum_{j=0}^n \frac{U_s^j-U_s^{j-1}}{\tau}\bigg\|_{L^q(\Ga^0)}^p \\
	\leq &\ \tau\sum_{n=0}^N   \bigg(\tau \sum_{j=0}^n1^{p'}\bigg)^{\frac{p}{p'}}\tau \sum_{j=0}^n \bigg\|\frac{U_s^j-U_s^{j-1}}{\tau} \bigg\|_{L^q(\Ga^0)}^p \\
	\leq &\ T^{p-1} \tau\sum_{n=0}^N 
	\bigg\|\bigg(\frac{U_s^j-U_s^{j-1}}{\tau}\bigg)_{j=0}^n \bigg\|_{L^{p}(L^q(\Ga^0))}^p .  
	\end{align*}
	Raising \eqref{eq:maxreg bound - pre power p} to power $p$ yields 
	\begin{align*}
	&\ \bigg\|\bigg(\frac{U_s^n-U_s^{n-1}}{\tau}\bigg)_{n=0}^N \bigg\|_{L^{p}(L^q(\Ga^0))}^p
	+ \|(U_s^n)_{n=0}^N\|_{L^{p}(W^{2,q}(\Ga^0))}^p \\
	\leq &\ C \tau\sum_{n=0}^N 
	\bigg\|\bigg(\frac{U_s^j-U_s^{j-1}}{\tau}\bigg)_{j=0}^n \bigg\|_{L^{p}(L^q(\Ga^0))}^p \\
	&\ +C\big( \| (F_s^n)_{n=k}^N\|_{L^{p}(L^q(\Ga^0))} 
	+ \tau^{-1}\|(U_s^i)_{i=0}^{k-1}\|_{L^{p}(L^q(\Ga^0))}+ \|(U_s^i)_{i=0}^{k-1}\|_{L^{p}(W^{2,q}(\Ga^0))}\big)^p .
	\end{align*}
	By a discrete Gronwall inequality we obtain the estimate \eqref{BDF-MaxLp-ss}. 
\end{proof}

By the change of variable $U^n(y)= U_s^n(X(y,s))$ the equation \eqref{BDF:Gamma-ss} can be equivalently written as 
\begin{align}
\label{BDF:Gamma-s0}
& \frac{1}{\tau} \sum_{j=0}^k\delta_j  a(\cdot,s) U^{n-j} 
- \nb_{\Ga^0}\cdot\big(B(\cdot,s)\nb_{\Ga^0} U^n\big)
= F^n \quad \text{on} \,\,\, \Ga^0 ,\quad k\le n\le T/\tau .
\end{align}
This, together with Lemma~\ref{Lemma:time-independent for BDF}, implies the following result. 
\begin{lemma}
	\label{lemma:MaxLp-s-BDF}
	For any given $s\in[0,T]$, the solutions of \eqref{BDF:Gamma-s0} 
	satisfy the discrete maximal $L^{p}$-regularity estimate for $0\le N\le T/\tau$:
	\begin{align}
	\label{MaxLp-s-BDF}
	&\ \bigg\|\bigg(\frac{U^n-U^{n-1}}{\tau}\bigg)_{n=0}^N \bigg\|_{L^{p}(L^q(\Ga^0))}
	+ \|(U^n)_{n=0}^N\|_{L^{p}(W^{2,q}(\Ga^0))}  \nonumber \\
	\leq &\ C \| (F^n)_{n=k}^N\|_{L^{p}(L^q(\Ga^0))}
	+  C(\tau^{-1}\|(U^i)_{i=0}^{k-1}\|_{L^{p}(L^q(\Ga^0))}+ \|(U^i)_{i=0}^{k-1}\|_{L^{p}(W^{2,q}(\Ga^0))}),
	\end{align}
	where $U^{-1}:=0$ and the term $\| (F_s^n)_{n=k}^N\|_{L^{p}(L^q(\Ga^0))} $ vanishes if $N<k$, and the constant $C>0$ is independent of $N$, $\tau$ and $s\in[0,T]$. 
\end{lemma}

{\bf Part II: Extension to evolving surfaces.}

Lemma \ref{lemma:MaxLp-s-BDF} requires the coefficients of \eqref{BDF:Gamma-s0} to be frozen at some fixed time $s\in[0,T]$. We therefore need to extend Lemma~\ref{lemma:MaxLp-s-BDF} to equation \eqref{eq:BDF - pulled back}, which has time-dependent coefficients. To this end, we use the perturbation arguments analogously as in the proof of Theorem \ref{theorem:MaxLp - Lagrangian} but in the time-discrete setting.

We rewrite \eqref{eq:BDF - pulled back} into the following form, with frozen coefficients on the left-hand side and perturbation terms on the right-hand side, for brevity omitting the spatial arguments in the coefficients, we obtain  
\begin{align}
\label{eq:BDF for perturbed linear PDE}
&\ \frac{1}{\tau}\sum_{j=0}^k\delta_j  a(t_N) U^{n-j} 
-\nb_{\Ga^0}\cdot\Big(B(t_N)\nb_{\Ga^0} U^n\Big) \nonumber \\
&=  a(t_n) F^n
+ \frac{1}{\tau}\sum_{j=0}^k \delta_j   (a(t_N) -a(t_{n-j}) )U^{n-j}      -\nb_{\Ga^0}\cdot\Big((B(t_N)-B(t_n))\nb_{\Ga^0} U^n\Big) \nonumber \\
&=: a(t_n) F^n + G^n + H^n .
\end{align}
Applying Lemma \ref{lemma:MaxLp-s-BDF} to the equation above, and using the bound \eqref{aXt} for $a(t_n)$, yields 
\begin{equation}
\label{eq:perturbed discr max reg estimate}
\begin{aligned}
&\ \bigg\|\bigg(\frac{U^n-U^{n-1}}{\tau}\bigg)_{n=0}^N \bigg\|_{L^{p}(L^q(\Ga^0))}
+ \|(U^n)_{n=0}^N\|_{L^{p}(W^{2,q}(\Ga^0))}  \\
\leq &\ C \|(F^n)_{n=k}^N\|_{L^{p}(L^q(\Ga^0))} + C\|(G^n)_{n=k}^N\|_{L^{p}(L^q(\Ga^0))} 
+C \|(H^n)_{n=k}^N\|_{L^{p}(L^q(\Ga^0))} \\
&\
+C(\tau^{-1}\|(U^i)_{i=0}^{k-1}\|_{L^{p}(L^q)}+ \|(U^i)_{i=0}^{k-1}\|_{L^{p}(W^{2,q}(\Ga^0))}) .
\end{aligned}
\end{equation}
The second and third terms on the right-hand side are bounded further separately. For the second term we obtain 
\begin{align*}
\|G^n\|_{L^q(\Ga^0)} 
&\leq 
\bigg\| \frac{1}{\tau}\sum_{j=0}^k\delta_j   (a(t_N) -a(t_{n-j}) )U^{n-j} \bigg\|_{L^q(\Ga^0)} \\
&=\bigg\|(a(t_N) -a(t_{n}) ) \frac{1}{\tau}\sum_{j=0}^k\delta_j   U^{n-j} 
+ \sum_{j=0}^k\delta_j   \frac{a(t_n) -a(t_{n-j})}{\tau} U^{n-j} \bigg\|_{L^q(\Ga^0)} \\
&\leq \big\| (a(t_N) -a(t_{n}) ) \overline\partial_\tau U^n \big\|_{L^q(\Ga^0)} 
+ \bigg\|\sum_{j=0}^k\delta_j   \frac{a(t_n) -a(t_{n-j})}{\tau} U^{n-j} \bigg\|_{L^q(\Ga^0)} \\
&\leq C|t_N-t_n| \big\| \overline\partial_\tau U^n \big\|_{L^q(\Ga^0)} 
+ C\sum_{j=0}^k\|U^{n-j} \|_{L^q(\Ga^0)} ,
\end{align*}
where we have used the estimate $\big|\frac{a(\cdot,t_n) -a(\cdot,t_{n-j})}{\tau}\big| \leq C$, following from the temporal smoothness of \eqref{def-ayt},
to derive the last inequality. 

\noindent For the third term similar estimates yield
\begin{align*}
\|H^n\|_{L^q(\Ga^0)} 
&\leq C|t_N-t_n|\|U^{n}\|_{W^{2,q}(\Ga^0)} + C\sum_{j=0}^k\|U^{n-j} \|_{L^q(\Ga^0)} .
\end{align*}
Substituting the estimates for $G^n$ and $H^n$ into \eqref{eq:perturbed discr max reg estimate} yields
\begin{align}
\label{eq:pre L inequality}
&\ \bigg\|\bigg(\frac{U^n-U^{n-1}}{\tau}\bigg)_{n=0}^N \bigg\|_{L^{p}(L^q(\Ga^0))}
+ \|(U^n)_{n=0}^N\|_{L^{p}(W^{2,q}(\Ga^0))} 
\nonumber \\
\leq &\ C \|(F^n)_{n=k}^N\|_{L^{p}(L^q(\Ga^0))} 
+ C\sum_{j=0}^k \|(U^{n-j})_{n=k}^N\|_{L^{p}(L^q(\Ga^0))} 
\nonumber \\
&\ + C\big(\|(|t_N-t_n| \overline\partial_\tau U^n)_{n=k}^N\|_{L^{p}(L^q(\Ga^0))} 
+ \|(|t_N-t_n|U^n)_{n=k}^N\|_{L^{p}(W^{2,q}(\Ga^0))}\big)
\nonumber \\
&\ + C(\tau^{-1}\|(U^i)_{i=0}^{k-1}\|_{L^{p}(L^q(\Ga^0))}+ \|(U^i)_{i=0}^{k-1}\|_{L^{p}(W^{2,q}(\Ga^0))})
\nonumber \\
\leq &\ C \|(F^n)_{n=k}^N\|_{L^{p}(L^q(\Ga^0))} 
+ C \|(U^n)_{n=0}^N\|_{L^{p}(L^q(\Ga^0))} 
\nonumber \\
&\ + C\big(\|(|t_N-t_n| \overline\partial_\tau U^n)_{n=k}^N\|_{L^{p}(L^q(\Ga^0))} 
+ \|(|t_N-t_n|U^n)_{n=k}^N\|_{L^{p}(W^{2,q}(\Ga^0))}\big) 
\nonumber \\
&\ + C(\tau^{-1}\|(U^i)_{i=0}^{k-1}\|_{L^{p}(L^q)}+ \|(U^i)_{i=0}^{k-1}\|_{L^{p}(W^{2,q}(\Ga^0))}) ,
\end{align}
which holds for $0\le N \le T/ \tau$. 

Analogously to the proof of Theorem~\ref{theorem:MaxLp - Lagrangian}, we define 
\begin{equation}
\label{Def-Lm}
\begin{aligned}
L^m =&\ \bigg\|\bigg(\frac{U^n-U^{n-1}}{\tau}\bigg)_{n=0}^N \bigg\|_{L^{p}(L^q(\Ga^0))}^p
+ \|(U^n)_{n=0}^N\|_{L^{p}(W^{2,q}(\Ga^0))}^p \\
= &\ \tau \sum_{n=0}^N \bigg( \bigg\|\frac{U^n-U^{n-1}}{\tau}\bigg\|_{L^q(\Ga^0)}^p + \|U^n\|_{W^{2,q}(\Ga^0)}^p\bigg) ,
\end{aligned}
\end{equation}
with $L^{-1} = 0$. 
Then \eqref{eq:pre L inequality} and summation by parts imply the following estimate for $L^n$: 
\begin{align*}
L^N 
\leq &\  C \|(F^n)_{n=k}^N\|_{L^{p}(L^q(\Ga^0))}^p + C \|(U^n)_{n=k}^N\|_{L^{p}(L^q(\Ga^0))}^p \\
&\ + C\|(|t_n-t_N|\overline\partial_\tau U^n)_{n=0}^N\|_{L^{p}(L^q(\Ga^0))}^p + C \|(|t_n-t_N|U^n)_{n=k}^N\|_{L^{p}(W^{2,q}(\Ga^0))}^p \\
&\ + C (\tau^{-1}\|(U^i)_{i=0}^{k-1}\|_{L^{p}(L^q(\Ga^0))}+ \|(U^i)_{i=0}^{k-1}\|_{L^{p}(W^{2,q}(\Ga^0))}) \\
\leq &\ C \|(F^n)_{n=k}^N\|_{L^{p}(L^q(\Ga^0))}^p  + C\|(U^n)_{n=0}^N\|_{L^{p}(L^q(\Ga^0))}^p \\
&\ + C \sum_{n=k}^N |t_n-t_N|^p \big(\|\overline\partial_\tau U^n\|_{L^q(\Ga^0)}^p + \|U^n\|_{W^{2,q}(\Ga^0)}^p\big) \\
&\ + C(\tau^{-1}\|(U^i)_{i=0}^{k-1}\|_{L^{p}(L^q(\Ga^0))}+ \|(U^i)_{i=0}^{k-1}\|_{L^{p}(W^{2,q}(\Ga^0))}) \\
\leq &\ C \|(F^n)_{n=k}^N\|_{L^{p}(L^q)}^p  + C\|(U^n)_{n=0}^N\|_{L^{p}(L^q(\Ga^0))}^p \\
&\ + C \sum_{n=0}^N |t_n-t_N|^p \bigg(\bigg\|\frac{U^n-U^{n-1}}{\tau}\bigg\|_{L^q(\Ga^0)}^p + \|U^n\|_{W^{2,q}(\Ga^0)}^p\bigg) 
\qquad\mbox{(using \eqref{BDFUn-dtUn} here)}\\
&\ + C(\tau^{-1}\|(U^i)_{i=0}^{k-1}\|_{L^{p}(L^q(\Ga^0))} + \|(U^i)_{i=0}^{k-1}\|_{L^{p}(W^{2,q}(\Ga^0))}) \\
= &\ C \|(F^n)_{n=k}^N\|_{L^{p}(L^q(\Ga^0))}^p + C \|(U^n)_{n=0}^N\|_{L^{p}(L^q(\Ga^0))}^p \\
&\ + C \sum_{n=0}^N |t_n-t_N|^p \big(L^n - L^{n-1}\big)	\qquad \mbox{($L^{-1}:=0$)} \\
&\ + C(\tau^{-1}\|(U^i)_{i=0}^{k-1}\|_{L^{p}(L^q(\Ga^0))} + \|(U^i)_{i=0}^{k-1}\|_{L^{p}(W^{2,q}(\Ga^0))}) \\
= &\ C \|(F^n)_{n=k}^N\|_{L^{p}(L^q(\Ga^0))}^p + C \|(U^n)_{n=0}^N\|_{L^{p}(L^q(\Ga^0))}^p \\
&\ + C\sum_{n=0}^N  \big(|t_{n-1}-t_N|^p - |t_{n}-t_N|^p\big) L^{n-1} \qquad \mbox{(summation by parts)} \\
&\ + C(\tau^{-1}\|(U^i)_{i=0}^{k-1}\|_{L^{p}(L^q(\Ga^0))}+ \|(U^i)_{i=0}^{k-1}\|_{L^{p}(W^{2,q}(\Ga^0))}) . 
\end{align*}
By the bound 
$$
|t_{n-1}-t_N|^p - |t_n-t_N|^p
= \tau \sum_{i=1}^p |t_n-t_N|^{p-i} |t_{n-1}-t_N|^{i-1} 
\leq \tau p T^{p-1} ,
$$
the last estimate is further bounded by
\begin{align*}
L^N
\leq &\ C \|(F^n)_{n=k}^N\|_{L^{p}(L^q(\Ga^0))}^p + C \|(U^n)_{n=0}^N\|_{L^{p}(L^q(\Ga^0))}^p \\
&\ + C(\tau^{-1}\|(U^i)_{i=0}^{k-1}\|_{L^{p}(L^q(\Ga^0))} +  \|(U^i)_{i=0}^{k-1}\|_{L^{p}(W^{2,q}(\Ga^0))}) + C \tau \sum_{n=0}^{N} L^{n-1} .
\end{align*}
Applying discrete Gronwall inequality, we obtain
\begin{align*}
L^N \leq &\ C \|(F^n)_{n=k}^N\|_{L^{p}(L^q(\Ga^0))}^p  + C\|(U^n)_{n=0}^N\|_{L^{p}(L^q(\Ga^0))}^p \\
&\ + C (\tau^{-1}\|(U^i)_{i=0}^{k-1}\|_{L^{p}(L^q(\Ga^0))}+ \|(U^i)_{i=0}^{k-1}\|_{L^{p}(W^{2,q}(\Ga^0))}) .
\end{align*}
Then substituting \eqref{Def-Lm} into the last inequality yields 
\begin{equation}
\label{eq:with extra term}
\begin{aligned}
&\ \bigg\|\bigg(\frac{U^n-U^{n-1}}{\tau}\bigg)_{n=0}^N \bigg\|_{L^{p}(L^q(\Ga^0))}
+ \|(U^n)_{n=0}^N\|_{L^{p}(W^{2,q}(\Ga^0))} \\
\leq &\ C \|(F^n)_{n=k}^N\|_{L^{p}(L^q(\Ga^0))}  + C\|(U^n)_{n=0}^N\|_{L^{p}(L^q(\Ga^0))} \\
&\ + C(\tau^{-1}\|(U^i)_{i=0}^{k-1}\|_{L^{p}(L^q(\Ga^0))} + \|(U^i)_{i=0}^{k-1}\|_{L^{p}(W^{2,q}(\Ga^0))}) ,
\end{aligned}
\end{equation}
which holds for $N \leq T/ \tau$, with a constant depending on $p$ and $T$, but independent of $N$ or $\tau$. 

Similarly as for \eqref{eq:maxreg bound - pre power p}, the low-order term $\|(U^n)_{n=0}^N\|_{L^{p}(L^q)}$ on the right-hand side is removed by a discrete Gronwall argument. 
This completes the proof of Theorem \ref{theorem:MaxLp - LagrangianBDF}.
\endproof

%

\section{Proof of Theorem \ref{theorem: convergence of BDF time discr}}
\label{Sec:TD-nonlinear}

In this section, we show stability and convergence of linearly implicit BDF methods for the nonlinear evolving surface PDE problem \eqref{eq:PDE} by using the discrete maximal $L^p$-regularity of BDF methods established in Theorem~\ref{theorem:MaxLp - LagrangianBDF}, 
in combination with a mathematical induction on the boundedness of numerical solutions in the $W^{1,\infty}$ norm, which will be derived by using the following discrete Sobolev embedding inequality, the time-discrete version of Lemma~\ref{lemma:space-time Sobolev inequality}. 

\begin{lemma}[Discrete inhomogeneous Sobolev embedding]
	\label{lemma:discrete space-time Sobolev inequality}
	{\it
		Let $1 < p,q < \infty$, satisfying $2/p+m/q<1$ with $\Ga^0 \subset \R^{m+1}$, and let $(w^n)_{n=0}^\infty \subset W^{2,q} (\Ga^0)$ be a sequence of functions (setting $w^{-1}=0$). Then
		\begin{align}\label{ineq-discr-embed}
		&\max_{0\le n\le N} \|w^n\|_{L^\infty(W^{1,\infty} (\Ga^0))}
		\le C\bigg(\bigg\|\bigg(\frac{w^n-w^{n-1}}{\tau}\bigg)_{n=0}^N\bigg\|_{L^p(L^q (\Ga^0))} + \|(w^n)_{n=0}^N\|_{L^p(W^{2,q} (\Ga^0))}\bigg) ,
		\end{align}
		where the constant $C>0$ is independent of $N \geq 0$.
	}
\end{lemma}
\begin{proof}
	In the case of planar domain, Lemma \ref{lemma:discrete space-time Sobolev inequality} was proved in \cite[Lemma 3.6]{Cai-Li-Lin-Sun-2019} as an extension of the continuous version of inhomogeneous Sobolev embedding in \cite[Lemma~3.1]{KunstmannLiLubich2016}. On the smooth surface $\Ga^0$, {\color{black}
as shown in \cite[p. 9]{Hebey-2000}, there exist $(\Omega_j,\varphi_j,\alpha_j)$, $j=1,\dots,m$, 
such that 	

1. $\{\Omega_j: j=1,\dots,m\}$ is an open coverning of $\Gamma^0$; 

2. $\varphi_j:\Omega_j\rightarrow M$ is a local chart for $j=1,\dots,m$; 

3. $\{\alpha_j: j=1,\dots,m\}$ is a partition of unity such that ${\rm supp}(\alpha_j)\subset \Omega_j$ for $j=1,\dots,m$.  

Then the Sobolev norm $\|w\|_{W^{k,p}(\Gamma^0)}$ is equivalent to 
$\sum_{j=1}^m \|(w\alpha_j)\circ\varphi_j\|_{W^{k,p}(\Omega_j)} $ for $1\le p\le \infty$. 
As a result, for any fixed $1<p,q<\infty$ such that $2/p+m/q<1$,  
\begin{align*}
&\max_{0\le n\le N} \|w^n\|_{L^\infty(W^{1,\infty} (\Ga^0))}^p \\ 
&\sim \sum_{j=1}^m \|(w^n\alpha_j)\circ\varphi_j\|_{W^{1,\infty} (\Omega_j)}^p \\ 
&\le C\sum_{j=1}^m\bigg(\bigg\|\bigg(\frac{(w^n\alpha_j)\circ\varphi_j-(w^{n-1}\alpha_j)\circ\varphi_j}{\tau}\bigg)_{n=0}^N\bigg\|_{L^p(L^q (\Omega_j))}^p + \|((w^n\alpha_j)\circ\varphi_j)_{n=0}^N\|_{L^p(W^{2,q} (\Omega_j))}^p \bigg) \\
&\le C\sum_{j=1}^m\sum_{n=0}^N\tau\bigg\|\Big(\frac{(w^n-w^{n-1}}{\tau}\alpha_j\Big)\circ\varphi_j\bigg\|_{L^q (\Omega_j)}^p + 
C\sum_{j=1}^m\sum_{n=0}^N\tau \|(w^n\alpha_j)\circ\varphi_j\|_{L^p(W^{2,q} (\Omega_j))}^p  \\
&\le C \sum_{n=0}^N\tau \bigg(\sum_{j=1}^m\bigg\|\Big(\frac{(w^n-w^{n-1}}{\tau}\alpha_j\Big)\circ\varphi_j\bigg\|_{L^q (\Omega_j)}\bigg)^p + 
C \sum_{n=0}^N\tau \bigg(\sum_{j=1}^m \|(w^n\alpha_j)\circ\varphi_j\|_{L^p(W^{2,q} (\Omega_j))}\bigg)^p \\
&\le C \sum_{n=0}^N\tau\bigg\|\frac{(w^n-w^{n-1}}{\tau}\bigg\|_{L^q (\Gamma^0)}^p + 
C \sum_{n=0}^N\tau \|w^n\|_{L^p(W^{2,q} (\Gamma^0))}^p \\
&\le C\bigg(\bigg\|\bigg(\frac{w^n-w^{n-1}}{\tau}\bigg)_{n=0}^N\bigg\|_{L^p(L^q (\Ga^0))} + \|(w^n)_{n=0}^N\|_{L^p(W^{2,q} (\Ga^0))}\bigg) ^p .
\end{align*}
This proves the result of Lemma \ref{lemma:discrete space-time Sobolev inequality}. 
}
\end{proof}

\subsection{Stability}

Let us first formulate the pull-back of the nonlinear problem \eqref{eq:PDE} onto the initial surface $\Ga^0$ similarly as \eqref{PDE:Gamma}, which yields the following problem:
\begin{align}
\label{PDE:Gamma - nonlinear}
\left\{
\begin{aligned}
\diff \Big(a(\cdot,t) U(\cdot,t) \Big) -\nb_{\Ga^0}\cdot\Big(B(\cdot,t)\nb_{\Ga^0} U(\cdot,t)\Big) =&\  a(\cdot,t) f\big(U(\cdot,t),K(\cdot,t) \nb_{\Ga^0} U(\cdot,t)\big) , \\ 
U(\cdot,0)=&\ u^0 ,
\end{aligned}
\right. 
\end{align}
where the matrix $K(y,t)$ is introduced in \eqref{Def-K-operator}. 
The linearly implicit BDF methods for \eqref{PDE:Gamma - nonlinear} reads 
\begin{equation}
\label{eq:BDF for nonlin}
\frac{1}{\tau} \sum_{j=0}^k \delta_j a^{n-j} U^{n-j} - \nb_{\Ga^0} \cdot \big( B^n \nb_{\Ga^0} U^n \big) = a^n f(\hat U^n, K^n \nb_{\Ga^0} \hat U^n) ,
\end{equation}
where we employed the notations 
$$
a^n = a(\cdot,t_n),\quad B^n = B(\cdot,t_n),\quad K^n = K(\cdot,t_n) \quad\mbox{and}\quad
U^n=u(X(\cdot,t_n),t_n), 
$$
and $\hat U^n = \sum^{k-1}_{j=0}\gamma_j U_{n-j-1} $. 

The exact solution $U_*^n = U(\cdot,t_n)$ of \eqref{PDE:Gamma - nonlinear} satisfies \eqref{eq:BDF for nonlin} up to a defect (consistency error) $D^n$, i.e.
\begin{equation}
\label{eq:perturbed nonlin BDF}
\frac{1}{\tau} \sum_{j=0}^k \delta_j a^{n-j} U_*^{n-j} - \nb_{\Ga^0} \cdot \big( B^n \nb_{\Ga^0} U_*^n \big) = a^n f(\widehat{U}_*^n, K^n \nb_{\Ga^0} \widehat{U}_*^n) + D^n \qquad n \geq k.
\end{equation}
We denote the errors between the numerical solution and the exact solution by $E^n = U^n - U_*^n$.
Upon subtracting \eqref{eq:perturbed nonlin BDF} from \eqref{eq:BDF for nonlin}, we obtain the following error equation for $n \geq k$:
\begin{equation}
\label{eq:nonlin error eq}
\frac{1}{\tau} \sum_{j=0}^k \delta_j a^{n-j} E^{n-j} - \nb_{\Ga^0} \cdot \big( B^n \nb_{\Ga^0} E^n \big) = a^n \Big(f(\hat u^n, K^n \nb_{\Ga^0} \hat u^n) - f(\widehat{U}_*^n, K^n \nb_{\Ga^0} \widehat{U}_*^n)\Big) - D^n .
\end{equation}

Using the discrete maximal parabolic $L^p$-regularity result in Theorem~\ref{theorem:MaxLp - LagrangianBDF}, we now prove that the error $E^n$ satisfies the following stability result.
\begin{proposition}[Stability]
	\label{proposition: stability for nonlinear - time discrete}
	Let $1<p,q<\infty$ be numbers satisfying $2/p+m/q<1$. There exist \bbk sufficiently small \ebk positive constants $\beta_0$ and $\tau_0$ (independent of $\tau$, but may depend on $T$), such that if $\tau\le \tau_0$ and the starting values and the defects satisfy the bounds 
	\begin{align}
	\tau^{-1} \| (U^i - U_*^i)_{i=0}^{k-1} \|_{L^{p}(L^q(\Ga^0))} +  \| (U^i - U_*^i)_{i=0}^{k-1} \|_{L^{p}(W^{2,q}(\Ga^0))}
	\leq &\ \beta_0 , \label{eq:assumed initial abound} \\
	\text{ and } \qquad \ \| (D^n)_{n=k}^N \|_{L^p(L^q(\Ga^0))} \leq &\ \beta_0 ,
	\label{eq:assumed defect bound}
	\end{align}
	Then  the error $E^n = U^n - U_*^n$ between the solutions of \eqref{eq:BDF for nonlin} and \eqref{eq:perturbed nonlin BDF} {\rm(}with $E^{-1}=0)$ satisfies the following stability estimates for $ N \leq T/\tau$: 
	\begin{align}
	\label{eq:maxreg stability bound for nonlinear - BDF}
	&\ \bigg\| \bigg(\frac{E^n-E^{n-1}}{\tau}\bigg)_{n=0}^N \bigg\|_{L^p(L^q(\Ga^0))} + \| (E^n)_{n=0}^N \|_{L^p(W^{2,q}(\Ga^0))} \nonumber \\
	&\qquad \leq C \| (D^n)_{n=k}^N \|_{L^p(L^q(\Ga^0))} + C (\tau^{-1}\|(E^i)_{i=0}^{k-1}\|_{L^{p}(L^q(\Ga^0))} + \|(E^i)_{i=0}^{k-1}\|_{L^{p}(W^{2,q}(\Ga^0))}) , \\[10pt]
	\label{eq:maxreg stability bound for nonlinear - BDF - infty version}
	&\ \max_{0\le n\le N} \| E^n \|_{W^{1,\infty}(\Ga^0)} \nonumber \\
	&\qquad \leq C \| (D^n)_{n=k}^N \|_{L^p(L^q(\Ga^0))} + C (\tau^{-1}\|(E^i)_{i=0}^{k-1}\|_{L^{p}(L^q(\Ga^0))} + \|(E^i)_{i=0}^{k-1}\|_{L^{p}(W^{2,q}(\Ga^0))}) .
	\end{align}
	The constants $\beta_0$, $\tau_0$ and $C$ are independent of $\tau$ and $N$, but may depend on the exact solution $U_*$ and $T$. 
\end{proposition}

\begin{proof}
	Similarly as in \cite{AkrivisLiLubich2017,KunstmannLiLubich2016}, let us assume that $M \leq T/\tau$ is the maximal integer such that the numerical solution satisfies the estimate
	\begin{equation}
	\label{eq:assumed bound}
	\begin{aligned}
	\| U^{n-1} \|_{W^{1,\infty}} 		\leq &\ \| U_* \|_{C([0,T];W^{1,\infty})} + 1
	\end{aligned}
	\quadfor n \leq M .
	\end{equation}
	\bbk Such an $M$ indeed exists since \eqref{eq:assumed bound} holds at least for the initial values (with $M=k$ for sufficiently small $\beta_0$). In fact, by the triangle inequality and the combination of \eqref{eq:assumed initial abound} with Lemma~\ref{lemma:discrete space-time Sobolev inequality}, we have  
	$$\max_{0\le i\le k-1} \| E^i \|_{W^{1,\infty}(\Ga^0)} \leq C \beta_0 < 1$$
	or $\beta_0$ sufficiently small. \ebk 
	
	We will now use discrete maximal $L^p$-regularity to show the stability bounds \eqref{eq:maxreg stability bound for nonlinear - BDF}--\eqref{eq:maxreg stability bound for nonlinear - BDF - infty version} for $m\leq M$, which would imply that $M$ is not maximal for \eqref{eq:assumed bound} unless $M=[T/\tau]$ (i.e. the maximal integer not strictly bigger than $T/\tau$).
	
	(a) Under assumption \eqref{eq:assumed bound}, by the local-Lipschitz continuity of $f$, we have 
	\begin{align}\label{local-Lipschitz-f1}
	|f(\hat U^n, K^n \nb_{\Ga^0} \hat U^n) - f(\widehat{U}_*^n, K^n \nb_{\Ga^0} \widehat{U}_*^n)|
	\le C(|\widehat E^n|+|\nabla \widehat E^n|) , \qquad n \leq M.
	\end{align}
	By applying Theorem \ref{theorem:MaxLp - LagrangianBDF} to the error equation \eqref{eq:nonlin error eq} up to $N \leq M$, and using \eqref{local-Lipschitz-f1}, we obtain
	\begin{equation}
	\label{eq:pre rhs estimate}
	\begin{aligned}
	&\ \bigg\| \bigg(\frac{E^n-E^{n-1}}{\tau}\bigg)_{n=0}^N \bigg\|_{L^p(L^q(\Ga^0))} + \| (E^n)_{n=0}^N \|_{L^p(W^{2,q}(\Ga^0))}  \\
	\leq &\  
	C \| (f(\hat u^n, K^n \nb_{\Ga^0} \hat u^n) - f(\widehat{U}_*^n, K^n \nb_{\Ga^0} \widehat{U}_*^n))_{n=k}^N \|_{L^p(L^q(\Ga^0))} 
	+ C \| (D^n)_{n=k}^N \|_{L^p(L^q(\Ga^0))} \\
	&\
	+C(\tau^{-1}\|(E^i)_{i=0}^{k-1}\|_{L^{p}(L^q(\Ga^0))} + \|(E^i)_{i=0}^{k-1}\|_{L^{p}(W^{2,q}(\Ga^0))}) \\
	\leq &\ C\| (\widehat E^n)_{n=k}^N \|_{L^p(W^{1,q}(\Ga^0))} 
	+ C \| (D^n)_{n=k}^N \|_{L^p(L^q(\Ga^0))} \\
	&\
	+C(\tau^{-1}\|(E^i)_{i=0}^{k-1}\|_{L^{p}(L^q(\Ga^0))} + \|(E^i)_{i=0}^{k-1}\|_{L^{p}(W^{2,q}(\Ga^0))})
	\\
	\leq &\ C\| (E^n)_{n=0}^N \|_{L^p(W^{1,q}(\Ga^0))} 
	+ C \| (D^n)_{n=k}^N \|_{L^p(L^q(\Ga^0))} \\
	&\ 
	+C(\tau^{-1}\|(E^i)_{i=0}^{k-1}\|_{L^{p}(L^q(\Ga^0))} + \|(E^i)_{i=0}^{k-1}\|_{L^{p}(W^{2,q}(\Ga^0))}) \\
	\leq &\ \varepsilon\| (E^n)_{n=0}^N \|_{L^p(W^{2,q}(\Ga^0))} +C\varepsilon^{-1}\| (E^n)_{n=0}^N \|_{L^p(L^q(\Ga^0))} 
	+ C \| (D^n)_{n=k}^N \|_{L^p(L^q(\Ga^0))} \\
	&\ 
	+C(\tau^{-1}\|(E^i)_{i=0}^{k-1}\|_{L^{p}(L^q(\Ga^0))} + \|(E^i)_{i=0}^{k-1}\|_{L^{p}(W^{2,q}(\Ga^0))})
	,
	\end{aligned}
	\end{equation}
	where we have used the Sobolev interpolation inequality (cf. \cite[Theorem 4.12]{Adams})
	$$
	\|E^n\|_{W^{1,q}(\Ga^0)}
	\le
	\varepsilon\|E^n\|_{W^{2,q}(\Ga^0)}
	+C\varepsilon^{-1}\|E^n\|_{L^q(\Ga^0)},
	$$
	with an arbitrary parameter $\varepsilon>0$. 
	By choosing $\varepsilon$ sufficiently small, the leading term $\varepsilon\| (E^n)_{n=0}^N \|_{L^p(W^{2,q}(\Ga^0))}$ can be absorbed to the left-hand side. Then, similarly as \eqref{eq:maxreg bound - pre power p}, the low-order term $C\|(E^n)_{n=0}^N\|_{L^{p}(L^q(\Ga^0))}$ on the right-hand side above can be removed by using Gronwall's inequality. 
	This proves the first estimate \eqref{eq:maxreg stability bound for nonlinear - BDF} for $m \leq M$.
	
	The $W^{1,\infty}(\Ga^0)$ error estimate can be obtained by applying Lemma \ref{lemma:discrete space-time Sobolev inequality}:
	\begin{equation}
	\label{eq:embedding}
	\max_{0\le n\le N} \| E^n \|_{W^{1,\infty}(\Ga^0)} \leq C\Bigg( \bigg\|\bigg(\frac{E^n - E^{n-1}}{\tau}\bigg)_{n=0}^N\bigg\|_{L^p(L^q(\Ga^0))} + \|(E^n)_{n=0}^N\|_{L^p(W^{2,q}(\Ga^0))} \Bigg) .
	\end{equation}
	Combining \eqref{eq:maxreg stability bound for nonlinear - BDF} and \eqref{eq:embedding} yileds \eqref{eq:maxreg stability bound for nonlinear - BDF - infty version} for $N \leq M$.
	
	(b) A triangle inequality and the above estimate imply 
	\begin{align*}
	\| U^M \|_{W^{1,\infty}(\Ga^0)} 
	\leq &\ \|U_*^M\|_{W^{1,\infty}(\Ga^0)}  + \| E^M \|_{W^{1,\infty}(\Ga^0)}  \\
	\leq &\ \|U_*^M\|_{W^{1,\infty}(\Ga^0)} + C \beta_0 
	\quad\mbox{(here \eqref{eq:assumed initial abound}--\eqref{eq:assumed defect bound} are used)} . 
	\end{align*}
	For a sufficiently small $\beta_0>0$ (independent of $M$ and $\tau$) this implies that the inequality in \eqref{eq:assumed bound} holds for $n=M+1$ as well. Thus $M$ cannot be maximal unless $M=[T/\tau]$. 
	Therefore, \eqref{eq:assumed bound} holds for $M=[T/\tau]$ and the estimates \eqref{eq:maxreg stability bound for nonlinear - BDF}--\eqref{eq:maxreg stability bound for nonlinear - BDF - infty version} was proved for all $N \leq T/\tau$.
\end{proof}

%
%

\subsection{Consistency}
\label{section: time discerete consistency}

In this section we show that the defect (or consistency error) $D^n$ obtained by inserting the exact solution into the BDF method \eqref{eq:perturbed nonlin BDF}, are bounded in the required norms by $C \tau^k$ for the linearly implicit BDF methods of order $k$. Hence, the assumed defect bounds in Proposition~\ref{proposition: stability for nonlinear - time discrete} are indeed satisfied.

\begin{lemma}
	\label{lemma: consistency - time discrete}
	Let the solution $u$, and hence also its pull-back $U_*$, be sufficiently smooth. Then the defect $D^n$ of the linearly implicit $k$-step BDF method, given by \eqref{eq:perturbed nonlin BDF}, is bounded by
	\begin{equation}
	\label{eq:consistency - time discrete}
	\max_{k \leq n  \leq T/\tau } \|D^n\|_{L^q(\Ga^0)} \leq C \tau^k ,
	\end{equation}
	where the constant $C>0$ only depends on $\GT$, the exact solution $u$ {\rm(}respectively $U)$ and $T$.
\end{lemma}

\begin{proof}
	Using the differential equation \eqref{PDE:Gamma - nonlinear}, recalling that its solution $U(\cdot,t_n) = U_*^n$, we can rewrite the defect $D^n$ from \eqref{eq:perturbed nonlin BDF} in the form
	\begin{align*}
	D^n =&\ \frac{1}{\tau} \sum_{j=0}^k \delta_j a(t_{n-j}) U(\cdot,t_{n-j}) - \frac{\partial }{\partial t} \big( a(\cdot,t_n) U(\cdot,t_n)\big)\\
	&\ + a(\cdot,t_n) \big[ f\big(U(\cdot,t_n),K(\cdot,t_n) \nb_{\Ga^0} U(\cdot,t_n)\big) -  f(\widehat{U}(\cdot,t_n), K(\cdot,t_n) \nb_{\Ga^0} \widehat{U}(\cdot,t_n)) \big].
	\end{align*}
	We then estimate the two pairs separately, using the standard techniques for consistency errors in BDF methods. Using the definition of the order of BDF methods (through the generating functions $\delta(\zeta)$ and $\gamma(\zeta)$ from \eqref{eq:BDF generating functions} and \eqref{eq:extrapolation - hat u}), together with Taylor expansion and bounded Peano kernels, in the same way as in \cite{AkrivisLiLubich2017,AkrivisLubich_quasilinBDF,LubichMansourVenkataraman_bdsurf}, we can obtain the pointwise bound
	\begin{equation*}
	\|D^n\|_{L^q(\Ga^0)} \leq C \tau^k , \quadfor n = k,\dotsc, [T/\tau] .
	\end{equation*}
\end{proof}

\subsection{Convergence}

The stability and consistency results proved in the last two subsections imply the following error estimates.
\begin{proof}[Proof of Theorem \ref{theorem: convergence of BDF time discr}]
	The error functions $e^n=u^n-u(\cdot,t_n)$ and $E^n=U^n-U(\cdot,t_n)$ are related through 
	$$
	e^n(X(\cdot,t_n))=E^n
	\quad\mbox{and}\quad 
	(\mathcal{F}_X(t_n,t_{n-1}) e^{n-1})(X(\cdot,t_n))=E^{n-1} .
	$$ 
	Lemma~\ref{lemma: consistency - time discrete} for the defects and the $O(\tau^k)$ assumption on the error of the starting values \eqref{initial-error-BDF} imply that the assumed bounds \eqref{eq:assumed initial abound}--\eqref{eq:assumed defect bound} of Proposition~\ref{proposition: stability for nonlinear - time discrete} are satisfied for sufficiently small step size $\tau$. 
	Then the stability estimates \eqref{eq:maxreg stability bound for nonlinear - BDF}--\eqref{eq:maxreg stability bound for nonlinear - BDF - infty version} imply
	\begin{align*}
	&\ \bigg( \tau \sum_{n=0}^N \Big( \bigg\|\frac{e^{n}-\mathcal{F}_X(t_n,t_{n-1}) e^{n-1}}{\tau}\bigg\|_{L^q(\Ga(t_n))}^p + \|e^{n}\|_{W^{2,q}(\Ga(t_n))}^p \Big) \bigg)^{\frac1p} \\
	\leq &\ C\bigg\| \bigg(\frac{E^{n}-E^{n-1}}{\tau}\bigg)_{n=k}^N\bigg\|_{L^p(L^q(\Ga^0))} + C\|(E^{n})_{n=k}^N\|_{L^p(W^{2,q}(\Ga^0))} \qquad \mbox{(pulled back to $\Ga^0$)} \\
	\leq &\ C \| (D^n)_{n=k}^N \|_{L^p(L^q(\Ga^0))} + C(\tau^{-1}\|(E^i)_{i=0}^{k-1}\|_{L^{p}(L^q(\Ga^0))} + \|(E^i)_{i=0}^{k-1}\|_{L^{p}(W^{2,q}(\Ga^0))}) \\
	\leq &\ C \tau^k . 
	\end{align*}
	and
	\begin{align*}
	\max_{0\le n\le N} \|e^{n}\|_{W^{1,\infty}(\Ga(t_n))} \leq C \tau^k . 
	\end{align*}
\end{proof}

\begin{remark} \upshape 
	The error estimate \eqref{Error-Time-discr} implies the boundedness of $\|u^{n}\|_{W^{1,\infty}(\Ga(t_n))}$, uniform with respect to the step size $\tau$. 
	After being pulled back to the initial surface, the estimate \eqref{Error-Time-discr-Lp} yields 
	\begin{align}\label{Lp-W2q-error}
	\|(U^{n}-U(\cdot,t_n))_{n=1}^N\|_{L^p(W^{2,q}(\Ga^0))} \le C\tau^k.
	\end{align}
	and 
	\begin{equation}
	\label{Uniform-RegU}
	\begin{aligned}
	&\ \bigg\|\bigg(\frac{U^{n}-U^{n-1}}{\tau}\bigg)_{n=1}^N\bigg\|_{L^p(L^q(\Ga^0))}
	+ \|(U^{n})_{n=1}^N\|_{L^p(W^{2,q}(\Ga^0))} \\
	\leq &\ C \bigg\| \bigg(\frac{E^{n}-E^{n-1}}{\tau}\bigg)_{n=1}^N\bigg\|_{L^p(L^q(\Ga^0))}
	+C\|(E^{n})_{n=1}^N\|_{L^p(W^{2,q}(\Ga^0))} \\
	&\ +\bigg\|\bigg(\frac{U(\cdot,t_{n})-U(\cdot,t_{n-1})}{\tau}\bigg)_{n=1}^N\bigg\|_{L^p(L^q(\Ga^0))}
	+ \|U(\cdot,t_{n})_{n=1}^N\|_{L^p(W^{2,q}(\Ga^0))} \\
	\leq &\ C \tau^k 		+\bigg\|\bigg(\frac{U(\cdot,t_{n})-U(\cdot,t_{n-1})}{\tau}\bigg)_{n=1}^N\bigg\|_{L^p(L^q(\Ga^0))}
	+ \|U(\cdot,t_{n})_{n=1}^N\|_{L^p(W^{2,q}(\Ga^0))} \\
	\leq &\ C .
	\end{aligned}
	\end{equation}
	Correspondingly, we have 
	\begin{align} \label{Time-discr-Reg}
	\bigg( 
	\tau \sum_{n=1}^N 
	\bigg\|\frac{u^{n}-\mathcal{F}_X(t_n,t_{n-1})u^{n-1}}{\tau}\bigg\|_{L^q(\Ga(t_n))}^p +\tau \sum_{n=1}^N  \|u^{n}\|_{W^{2,q}(\Ga(t_n))}^p 
	\bigg)^{\frac1p} 
	\le &\ C.
	\end{align}
	The estimate \eqref{Lp-W2q-error} further implies 
	\begin{align}
	\bigg\|\bigg(\frac{1}{\tau}\sum_{j=0}^k\delta_j(U^{n-j}-U(\cdot,t_{n-j}))\bigg)_{n=1}^N\bigg\|_{L^p(W^{2,q}(\Ga^0))} \le C\tau^{k-1}\le C.
	\end{align}
	By using the triangle inequality, we have 
	\begin{align}\label{discr-LpH2U}
	&\ \bigg\|\bigg(\frac{1}{\tau}\sum_{j=0}^k\delta_j U^{n-j}\bigg)_{n=k}^N\bigg\|_{L^p(W^{2,q}(\Ga^0))} \nonumber \\
	\leq &\ \bigg\|\bigg(\frac{1}{\tau}\sum_{j=0}^k\delta_j (U^{n-j}-U(\cdot,t_{n-j}))\bigg)_{n=k}^N\bigg\|_{L^p(W^{2,q}(\Ga^0))}
	+\bigg\|\bigg(\frac{1}{\tau}\sum_{j=0}^k\delta_j U(\cdot,t_{n-j})\bigg)_{n=k}^N\bigg\|_{L^p(W^{2,q}(\Ga^0))} \nonumber\\
	\leq &\ C +\|\partial_tU\|_{L^p(0,T;W^{2,q}(\Ga^0))} \leq C .
	\end{align}
\end{remark}

\bbk 
\section{Numerical experiments}
\label{section:numerics}

To support the theoretical results in Theorem \ref{theorem: convergence of BDF time discr}, we present some numerical results with linear evolving surface finite elements and BDF methods of order $k$ (specified later). 
Quadratures of sufficiently high order were used to compute the finite element vectors and matrices so that the resulting quadrature error does not feature in the discussion of the accuracies of the schemes. 
The parametrisation of the finite elements was inspired by \cite{BCH2006}. 
The initial meshes were all generated using DistMesh by \cite{distmesh}, without taking advantage of any symmetry of the surface.

%


We consider the parabolic problem \eqref{eq:PDE} with the following nonlinear source term: 
\begin{equation*}
	f(u,\nbg u) = |\nbg u|^2 u + \rho
\end{equation*}
i.e. which corresponds to the harmonic map heat flow (cf.~\cite{ElliottFritz2016}) with an additional remainder $\rho$, on the evolving surface given by
\begin{equation*}
	\Ga\t = \big\{ x \in \R^3 \ \mid \ a(t)^{-1} x_1^2 + x_2^2 + x_3^2 - 1 = 0 \big\} ,
\end{equation*}
where $a(t)=1+\frac{1}{4}\sin(2\pi t)$, see \cite{DziukElliott_ESFEM}, and also \cite{DziukLubichMansour_rksurf,KovacsPower_max,KovacsPower_quasilinear,highorderESFEM}.
The initial value and the inhomogeneity $\rho$ are chosen corresponding to the exact solution $u(x,t) = e^{-t} x_1 x_2$.
In the experiments we use a sequence of meshes (with roughly doubling degrees of freedom, see plots) and for a sequence of time steps $\tau_{k+1} = \tau_k / 2$ with $\tau_0 = 0.2$.

In order to illustrate the convergence results \eqref{Error-Time-discr} of Theorem~\ref{theorem: convergence of BDF time discr} we have computed the $L^\infty(W^{1,\infty})$ norm of the errors between the fully discrete numerical solution and (the nodal interpolation of the) exact solution. The $L^\infty(W^{1,\infty})$ norm is computed by evaluating the (elementwise linear) fully discrete numerical solution and its gradient on each element. Note that for linear evolving surface finite element solutions the norm in \eqref{Error-Time-discr-Lp} does not make sense globally on $\Ga_h(t_n)$.
The initial values are chosen to be the nodal interpolation of the exact initial values $u(\cdot,t_i)$, for $i=0,\dotsc,k-1$.

In Figure~\ref{fig:maxreg BDF}, left- and right-hand sides, we report the $L^\infty(W^{1,\infty})$ norm of the errors of the second and fourth order BDF method, respectively, over the time interval $[0,T]$ with $T=1$.
The log-log plots of Figure~\ref{fig:maxreg BDF} report on the $L^\infty(W^{1,\infty})$ norm of the errors against against the time step size $\tau$.
In both plots, the lines marked with different symbols and different colours correspond to different mesh widths, while on each line a marker corresponds to a time step sizes $\tau_k$.
We can observe two regions in the figures: a region where the temporal discretisation error dominates, matching the $O(\tau^k)$ order of the BDF method ($k=2$ and $4$), see \eqref{Error-Time-discr} of Theorem~\ref{theorem: convergence of BDF time discr}, (see the reference lines), and a region, with small time step size, where the spatial discretisation error dominates (the error curves flatten out).

\begin{figure}[htbp]
	\centering
	\includegraphics[width=0.45\textwidth]{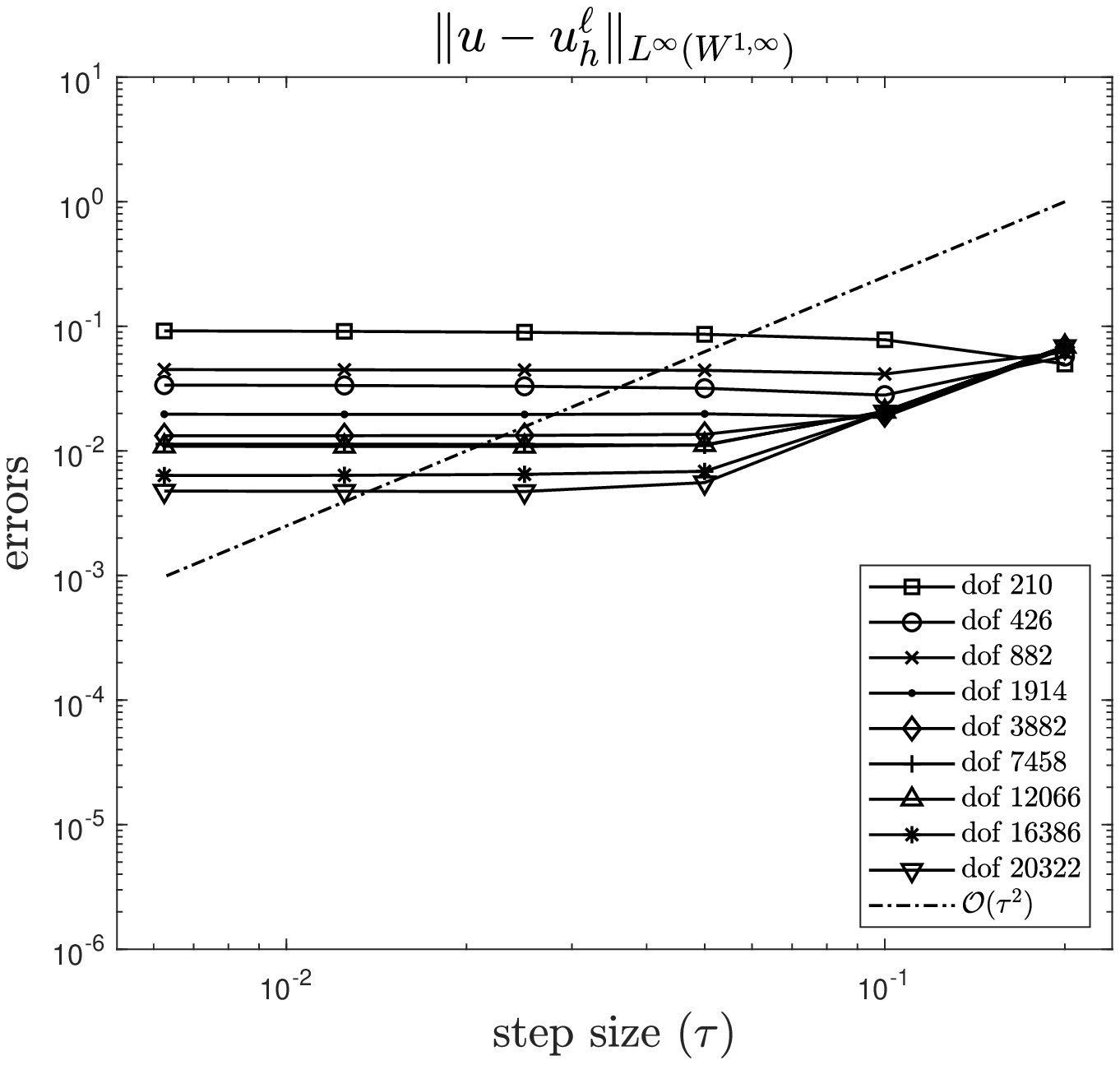}
	\includegraphics[width=0.45\textwidth]{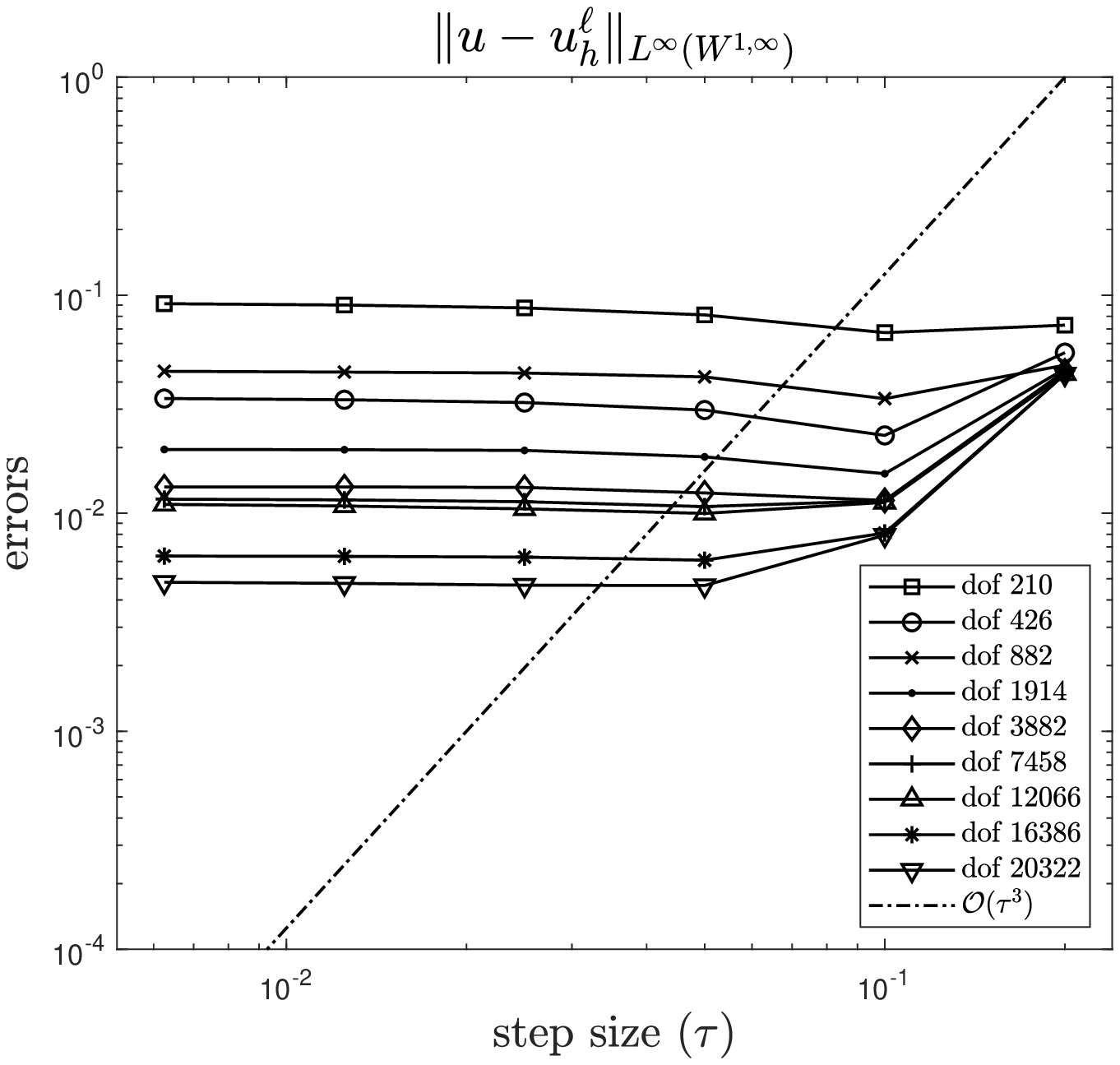}
	\caption{$L^\infty(W^{1,\infty})$ norm of the errors for a  BDF2 and BDF3 / linear ESFEM discretisations for the evolving surface PDE.}\label{fig:maxreg BDF}
\end{figure}

\ebk 

\section{Conclusions}

We have established the basic maximal $L^p$-regularity results of BDF methods for evolving surface parabolic PDEs. 
By using these results, we have proved optimal-order convergence of BDF methods for nonlinear evolving surface PDEs with a general nonlinear term which is not necessarily globally Lipschitz continuous. 
The maximal $L^p$-regularity results of BDF methods established in this paper and the techniques to handle locally Lipschitz continuous nonlinearities in evolving surface PDEs may be used to remove the grid ratio conditions $\tau=O(h^\kappa)$ for mean curvature flow and Willmore flow in \cite{MCF,Willmore} through analyzing the temporal semi-discretization and full discretization separately as in \cite{LiSun2013}.

\section*{Funding}

In the initial phase of this work, the authors were partially supported by a grant from the Germany/Hong Kong Joint Research Scheme sponsored by Research Grants Council of Hong Kong and the German Academic Exchange Service (G-PolyU502/16). 

The work of Bal\'azs Kov\'acs is supported by the Heisenberg Programme of the Deutsche Forschungsgemeinschaft (DFG, German Research Foundation) -- Project-ID 446431602.
The work of Buyang Li is supported by Research Grants Council of Hong Kong (GRF project 15300920).

\section*{Acknowledgement}
We thank the anonymous referees for the valuable comments and suggestions. 


\appendix

\section*{Appendix: Expressions of $a(X,t)$, $B(X,t)$ and $K(x,t)$ in a local chart}
\label{sec:Appendix}

For each point \(x\in \surface(0)\) there exists an open ball $D\subset\R^m$ and a smooth map \(\vphi\colon D \to \R^{m+1}\) (called local parametrization) such that \(\vphi\) is bijective and the $(m+1)\times m$ Jacobian matrix $\nabla \vphi$ has rank m, and \(x\in\vphi(D) \subset \surface\). 
By using the abbreviation $X_t:=X(\cdot,t)$, the composition map
$X_t\circ\vphi:D\rightarrow \Ga(t)\subset\R^{m+1}$ is a local chart of $\Ga(t)$,
with the Riemannian metric tensor 
\begin{align}\label{metric-tensor}
g_{ij}(\xi,t)=
\frac{\partial (X_t\circ\vphi)(\xi)}{\partial \xi_i}
\cdot \frac{\partial (X_t\circ\vphi)(\xi)}{\partial \xi_j}=
\sum_{m=1}^3\frac{\partial (X_t\circ\vphi)_m(\xi)}{\partial \xi_i}
\frac{\partial (X_t\circ\vphi)_m(\xi)}{\partial \xi_j},
\quad\forall\,\xi\in D. 
\end{align}

If we define $g=\det(g_{ij})$, then
the surface area element $\d S_t$ on $\Ga(t)$ can be expressed as
$$
\d S_t((X_t\circ\vphi)(\xi))=\sqrt{g(\xi,t)}\d \xi .
$$
Hence, for two functions $u(\cdot,t)$ and $\varphi(\cdot,t)$ defined on $\Ga(t)$, 
\begin{align}
\begin{aligned}
\int_{ (X_t\circ \vphi)(D)\cap \Ga(t) } u \vphi 
&=\int_{D} (u(\cdot,t) \circ X_t\circ\vphi) (\xi) \, (\varphi(\cdot,t) \circ X_t\circ\vphi) (\xi)
\sqrt{g(y,t)}\d y \\
&=\int_{D} (u(\cdot,t) \circ X_t)(\vphi(\xi)) \,  (\varphi(\cdot,t) \circ X_t)(\vphi(\xi))
\sqrt{\frac{g(\xi,t)}{g(\xi,0)}} \sqrt{g(\xi,0)}\d \xi \\
&=\int_{\vphi(D)\cap \Ga(0)}  u(X_t(y),t) \varphi(X_t(y),t)
\sqrt{\frac{g(\vphi^{-1}(y),t)}{g(\vphi^{-1}(y),0)}}   \, ,
\end{aligned}
\end{align}
where we have used the change of variable $y=\vphi(\xi)$ in the last equality. 
By comparing \eqref{eq:ES-PDE-weak-form} and \eqref{eq:ES-PDE-weak-form-2}, we obtain
\begin{align}\label{def-ayt}
a(y,t) =\sqrt{\frac{g(\vphi^{-1}(y),t)}{g(\vphi^{-1}(y),0)}}
\end{align}

It is easy to check that the function $a(y,t)$ given by \eqref{def-ayt} is well-defined, i.e.,
independent of the choice of the local parametrization \(\vphi\). 
In fact, if we choose another  local parametrization \(\widetilde\vphi:\widetilde D\rightarrow\R^{m+1}\) then 
\begin{align*}
\widetilde g(\widetilde\vphi^{-1}(y),t)
&=
\det\bigg(\sum_{r=1}^{m+1}\frac{\partial (X_t\circ\widetilde\vphi)_r}{\partial \widetilde\xi_i}
\frac{\partial (X_t\circ\widetilde\vphi)_r}{\partial \widetilde\xi_j}\bigg) \\
&=
\det\bigg[\sum_{r=1}^{m+1}\bigg(\sum_{k=1}^2\frac{\partial (X_t\circ \vphi)_r}{\partial \xi_k}
\frac{\partial \xi_k}{\partial \widetilde\xi_i}\bigg)
\bigg(\sum_{l=1}^2\frac{\partial (X_t\circ \vphi)_r}{\partial \xi_l}
\frac{\partial \xi_l}{\partial \widetilde\xi_j}\bigg)  \bigg]\\
&=\det\bigg[\sum_{l=1}^m\sum_{k=1}^m\frac{\partial \xi_k}{\partial \widetilde\xi_i}
\bigg(\sum_{m=1}^3 \frac{\partial (X_t\circ \vphi)_m}{\partial \xi_k}
\frac{\partial (X_t\circ \vphi)_m}{\partial \xi_l}\bigg)
\frac{\partial \xi_l}{\partial \widetilde\xi_j} \bigg]\\
&=\det\bigg(\frac{\partial \xi_k}{\partial \widetilde\xi_i}\bigg)
\det\bigg(\sum_{r=1}^{m+1} \frac{\partial (X_t\circ \vphi)_r}{\partial \xi_k}
\frac{\partial (X_t\circ \vphi)_r}{\partial \xi_l}\bigg)
\det\bigg( \frac{\partial \xi_l}{\partial \widetilde\xi_j} \bigg)\\
&=\det\bigg(\frac{\partial \xi_k}{\partial \widetilde\xi_i}\bigg)^2
g(\vphi^{-1}(y),t)
\end{align*}
and similarly,
\begin{align*}
\widetilde g(\widetilde\vphi^{-1}(y),0)
&=\det\bigg(\frac{\partial \xi_k}{\partial \widetilde\xi_i}\bigg)^2
g(\vphi^{-1}(y),0),
\end{align*}
which imply 
\begin{align*}
\widetilde a(\widetilde y,t)
=
\sqrt{\frac{\widetilde g(\widetilde \vphi^{-1}(y),t)}{\widetilde g(\widetilde \vphi^{-1}(y),0)}}
=\sqrt{\frac{\det\bigg(\dfrac{\partial \xi_k}{\partial \widetilde\xi_i}\bigg)^2g(\vphi^{-1}(y),t)}{\det\bigg(\dfrac{\partial \xi_k}{\partial \widetilde\xi_i}\bigg)^2g(\vphi^{-1}(y),0)}}
=\sqrt{\frac{g(\vphi^{-1}(y),t)}{g(\vphi^{-1}(y),0)}}
=a(y,t).
\end{align*}

Similarly, it is well known that the tangential gradient $\nabla_{\Ga(t)}u$ can be expressed in the local chart by \cite[equation (3.1.17)]{JostGeometry}
\begin{align}\label{nabla-u-local}
&(\nabla_{\Ga(t)}u)\circ  (X_t\circ\vphi )
=\sum_{i,j=1}^m g^{ij}(\cdot,t)\frac{\partial (u\circ X_t\circ\vphi) }{\partial \xi_j}  \,
\frac{\partial (X_t\circ\vphi)}{\partial \xi_i} ,\\
&(\nabla_{\Ga(0)}(u\circ X_t))\circ  \vphi
=\sum_{i,j=1}^m g^{ij}(\cdot,0)\frac{\partial (u\circ X_t\circ\vphi) }{\partial \xi_j}  \,
\frac{\partial  \vphi }{\partial \xi_i} .
\label{nabla-u-local22}
\end{align}
where $g^{ij}$ is the inverse matrix of $g_{ij}$ defined in \eqref{metric-tensor}. 
As a result, we have 
\begin{align}
&[(\nabla_{\Ga(t)}u)\circ  (X_t\circ\vphi) ]\cdot
\frac{\partial (X_t\circ\vphi)}{\partial \xi_k}
=\frac{\partial (u\circ X_t\circ\vphi) }{\partial \xi_k} , \\
&[(\nabla_{\Ga(0)}(u\circ X_t))\circ  \vphi ] \cdot
\frac{\partial   \vphi }{\partial \xi_k}
= \frac{\partial (u\circ X_t\circ\vphi) }{\partial \xi_k}
\end{align}
By using the identities above, we have
\begin{align}
\begin{aligned}
&\int_{ (X_t\circ \vphi)(D)\cap \Ga(t) } \nabla_{\Ga(t)}u\cdot \nabla_{\Ga(t)}\vphi \,\d S_t \\
&=\int_{D} 
\sum_{i,j=1}^2g^{ij}(\cdot,t)\frac{\partial (u\circ X_t\circ\vphi)}{\partial \xi_j}  \,
\frac{\partial (X_t\circ\vphi)}{\partial \xi_i}
\sum_{k,l=1}^2g^{kl}(\cdot,t)\frac{\partial (\varphi\circ X_t\circ\vphi)}{\partial \xi_l}  \,
\frac{\partial (X_t\circ\vphi)}{\partial \xi_k}
\sqrt{g(\cdot,t)} \, \d \xi \\
&=\int_{D} 
\sum_{i,j=1}^2g^{ij}(\cdot,t)\frac{\partial (u\circ X_t\circ\vphi)}{\partial \xi_j}  \,
\sum_{k,l=1}^2g^{kl}(\cdot,t)\frac{\partial (\varphi\circ X_t\circ\vphi)}{\partial \xi_l}  \,
g_{ik}(\cdot,t) \sqrt{g(\cdot,t)} \, \d\xi \\
&=\int_{D} 
\sum_{i,j=1}^2\sqrt{g(\cdot,t)} g^{ij}(\cdot,t)
\frac{\partial (u\circ X_t\circ\vphi)}{\partial \xi_i}  \, \frac{\partial (\varphi\circ X_t\circ\vphi)}{\partial \xi_j}  \, \d \xi \,  \\
&=\int_{D} 
\sum_{i,j=1}^2\sqrt{g(\cdot,t)} g^{ij}(\cdot,t)
\bigg([(\nabla_{\Ga(0)}(u \circ X_t))\circ  \vphi ] \cdot
\frac{\partial \vphi }{\partial \xi_i}  \bigg) \,
\bigg([(\nabla_{\Ga(0)}(\varphi \circ X_t))\circ  \vphi ] \cdot
\frac{\partial   \vphi }{\partial \xi_j}  \bigg) \, \d\xi \,  \\
&=\int_{D} 
A(\xi,t)[(\nabla_{\Ga(0)}(u \circ X_t))\circ  \vphi(\xi) ]  \cdot [(\nabla_{\Ga(0)}(\varphi \circ X_t))\circ  \vphi(\xi) ]  \, \sqrt{g(\xi,0)} \d \xi \,  \\
&=\int_{\vphi(D)\cap \Ga(0)}
B(\cdot ,t) \nabla_{\Ga(0)}(u\circ X_t)     \cdot \nabla_{\Ga(0)}(\varphi\circ X_t) 
\d S_0 \, ,
\end{aligned} 
\end{align}
with
\begin{align}  \label{exp-Axit}
A(\xi,t)
&=\sum_{i,j=1}^m\sqrt{\frac{g(\xi,t)}{g(\xi,0)}} g^{ij}(\xi,t)
\frac{\partial   \vphi (\xi)}{\partial \xi_i} \otimes
\frac{\partial   \vphi(\xi) }{\partial \xi_j}
\end{align}
and $B(\cdot,t)=A(\cdot,t)\circ \vphi^{-1}$, i.e., 
\begin{align}
\label{eq:1}
B(y,t)
&=\sqrt{\frac{g(\vphi^{-1}(y),t)}{g(\vphi^{-1}(y),0)}}\sum_{i,j=1}^m g^{ij}(\vphi^{-1}(y),t)
(\partial_{\xi_i} \vphi) (\vphi^{-1}(y))  \otimes
(\partial_{\xi_j} \vphi) (\vphi^{-1}(y)) ,
\quad\forall\, y\in \Ga(0). 
\end{align}
Again, \(B(y,t)\) is well-defined and independent of the choice of the local parametrization (the proof is similar as that for $a(y,t)$ below \eqref{def-ayt}). 

By fixing a local parametrization $\vphi:D\rightarrow \R^{m+1}$ of the surface $\Ga(0)$, it is straightforward to verify the positivity and smoothness of $a(y,t)$ and $B(y,t)$ at a fixed point $y\in\Ga(0)$. The lower and upper bounds in \eqref{aXt}--\eqref{BXt} are consequences of the compactness of surface $\Ga(0)$. 

Let $K(y,t):T_y\rightarrow T_{X(y,t)}$ be a linear operator defined by 
\begin{align}\label{Def-Kyt}
K(y,t)\bigg[\frac{\partial  \vphi }{\partial \xi_i}\circ \vphi^{-1}(y)\bigg]=\frac{\partial (X_t\circ\vphi) }{\partial \xi_i} \circ \vphi^{-1}(y) ,\quad i=1,\dots,m,
\end{align}
where $\big\{\frac{\partial  \vphi }{\partial \xi_1}\circ \vphi^{-1}(y),\frac{\partial  \vphi }{\partial \xi_2}\circ \vphi^{-1}(y)\big\}$ is a basis for the tangent space $T_y$ at $y\in\Ga(0)$, and  $\big\{\frac{\partial  (X_t\circ\vphi) }{\partial \xi_1}\circ\vphi^{-1}(y),\frac{\partial  (X_t\circ\vphi) }{\partial \xi_2}\circ\vphi^{-1}(y)\big\}$ a basis for the tangent space $T_{X(y,t)}$ at $X(y,t)\in\Ga(t)$. 
Then \eqref{nabla-u-local}--\eqref{nabla-u-local22} imply that 
\begin{align}\label{Def-KK}
K(y,t) \nabla_{\Ga(0)}(u\circ X_t) = \nabla_{\Ga(t)}u .
\end{align}
Since the matrix of the linear operator $K(y,t)$ under this choice of bases is the identity matrix, it follows that the operator $K(y,t)$ is smooth and invertible.

\bibliographystyle{IMANUM-BIB}
\bibliography{maxreg_bib}

\end{document}